\documentclass[preprint,12pt]{article}
\usepackage[top=1.1in, bottom=1.2in, left=0.8in, right=0.8in]{geometry}
\usepackage{amsmath,amsfonts,amssymb,amsthm}
\usepackage{mathrsfs,dsfont}
\usepackage{graphicx}
\usepackage{colortbl,dcolumn}
\usepackage{psfrag}
\usepackage{booktabs}
\usepackage{cite}
\usepackage{subfigure}
\usepackage{hyperref}
\allowdisplaybreaks[4]
\numberwithin{equation}{section}

\newcommand{\R}{\mathbb{R}}
\newcommand{\N}{\mathbb{N}}
\newcommand{\E}{\mathbb{E}}
\renewcommand{\P}{\mathbb{P}}
\newcommand{\tr}{\operatorname{trace}}

\newcommand{\diff}{{\,\rm{d}}}
\newcommand{\dif}{{\rm{d}}}

\newcommand{\F}{\mathcal{F}}

\newtheorem{theorem}{Theorem}[section]

\newtheorem{lemma}[theorem]{Lemma}

\newtheorem{assumption}[theorem]{Assumption}

\begin{document}
	
\title{Strong and weak convergence orders of numerical methods for SDEs driven by time-changed L\'{e}vy noise\footnotemark[1]}
	
\footnotetext{\footnotemark[1] This work was supported by National Natural Science Foundation of China (No. 12201552) and Yunnan Fundamental Research Projects (No. 202301AU070010).}

\author{Ziheng Chen\footnotemark[2],\quad
		Jiao Liu\footnotemark[3],\quad
		Anxin Wu\footnotemark[4]}

\footnotetext{\footnotemark[2] School of Mathematics and Statistics, Yunnan University, Kunming, Yunnan, 650500, China. Email: czh@ynu.edu.cn.}
	
\footnotetext{\footnotemark[3] School of Mathematics and Statistics, Yunnan University, Kunming, Yunnan, 650500, China. Email: liujiao3@stu.ynu.edu.cn}
	
\footnotetext{\footnotemark[4] School of Mathematics and Statistics, Yunnan University, Kunming, Yunnan, 650500, China. Email: wuanxin@stu.ynu.edu.cn. Corresponding author.}
	
\date{}
\maketitle

\begin{abstract}
      {\rm\small This work investigates the strong and weak convergence orders of numerical methods for SDEs driven by time-changed L\'{e}vy noise under the globally Lipschitz conditions. Based on the duality theorem, we prove that the numerical approximation generated by the stochastic $\theta$ method with $\theta \in [0,1]$ and the simulation of inverse subordinator converges strongly with order $1/2$. Moreover, the numerical approximation combined with the Euler--Maruyama method and the estimate of inverse subordinator is shown to have the weak convergence order $1$ by means of the Kolmogorov backward partial integro differential equations. These theoretical results are finally confirmed by some numerical experiments.} 
      \\
      

		
\end{abstract}
\date{}
\maketitle

\section{Introduction}
Stochastic differential equations (SDEs) driven by Gaussian noise and L\'{e}vy noise are powerful tools to model systems influenced by both continuous and jump-like random perturbations, having wide applications in finance, physics, information theory, statistical mechanics and many other scientific fields (see, e.g., \cite{applebaum2009levy, cont2004financial, mao2008stochastic, mccauley2013stochastic}). These classical SDEs usually assume that the underlying noise processes evolve according to the natural time scale with independent increments. Nevertheless, Gaussian noise and L\'{e}vy noise fall short of capturing the full complexity and richness of the observed real data for many practical scenarios, which do not evolve uniformly over time and often contain dependent increments. For example, market volatility may be more intense in certain periods and more stable in others \cite{ni2020study}. In order to capture irregular temporal dynamics, it is necessary and significant to consider SDEs driven by time-changed Gaussian noise and L\'{e}vy noise. As such equations not only offer greater flexibility in modeling complex systems but also exhibit a close connection to fractional Fokker--Planck equations, they have emerged as one of the most active and rapidly developing areas in stochastic analysis over the past few decades; see, e.g., \cite{magdziarz2009stochastic, nane2016stochastic, nane2017stability, nane2018path, nane2021timechanged, ni2020study} and references therein. Noting that the exact solutions of SDEs driven by time-changed Gaussian noise and L\'{e}vy noise are rarely available, numerical solutions thus become a powerful tool to understand the behavior of the underlying equations.

Recently, much progress has been made in the design and analysis of numerical approximations for SDEs driven by time-changed Gaussian noise. Base on the duality principle, the seminal work \cite{jum2016strong} establishes the strong and weak convergence orders of Euler--Maruyama method for a class of non-autonomous SDEs with time-space-dependent coefficient under the global Lipschitz condition and the linear growth condition. Taking into account that most SDEs arising from applications possess super-linearly or sub-linearly growing coefficients, \cite{deng2020semiimplicit} investigates the strong convergence order and mean square polynomial stability of the semi-implicit Euler--Maruyama method for SDEs whose drift coefficient can grow super-linearly 
and the diffusion coefficient obeys the global Lipschitz condition. For SDEs with super-linearly growing drift and diffusion coefficients, \cite{li2021truncated, liu2020truncated} prove the strong convergence order of the truncated Euler--Maruyama method. In order to achieve higher order accuracy, \cite{liu2023milstein} proposes a Milstein-type method for time-changed SDEs, whose spatial variables in the coefficients satisfy the super-linear growth condition and the temporal variables obey some H\"{o}lder’s continuity condition. Furthermore, When the considered time-changed SDEs satisfy the local Lipschitz condition, \cite{chen2025strong} studies the strong convergence and asymptotical mean square stability of stochastic theta method. As is seen from the aforementioned references, the strong convergence numerical methods for SDEs driven by time-changed Gaussian noise has been well studied. However, to the best of our knowledge, it remains unknown whether a numerical method applied to SDEs driven by time-changed L\'{e}vy noise converges strongly.

This work aims to establish the strong and weak convergence of numerical 
approximations for the following SDEs driven by time-changed L\'{e}vy noise
\begin{align}\label{eq:TCSDE}
	  \diff{X(t)}
	  =&~
	  f(E(t),X(t-))\diff{E(t)}
	  +
	  g(E(t),X(t-))\diff{W(E(t))} \notag
	  \\&~+ 
      \int_{|z|<c} h(E(t),X(t-),z)  
	  \tilde{N}(\dif{z},\dif{E(t)}), 
	  \quad t \in [0,T]
\end{align}
with $X(0) = x_{0} \in \R^{d}$, where $c > 0$, $X(t-) := \lim_{s \uparrow t}X(s), t > 0$, and the coefficients $f \colon \R^{+} \times \R^{d} \to \R^{d}, g  \colon \R^{+} \times \R^{d} \to \R^{d \times m}$ and $h \colon  \R^{+} \times \R^{d} \times \R \to \R^{d}$ are Borel measurable functions. Besides, $\{W(t)\}_{t \geq 0}$ is an $m$-dimensional standard Brownian motion on the complete filtered probability space $(\Omega,\F,\P,\{\F_{t}\}_{t \geq 0})$ with the filtration $\{\F_{t}\}_{t \geq 0}$ satisfying the usual conditions. Let $\tilde{N}(\dif{z},\dif{t}) := N(\dif{z},\dif{t}) - \nu(\dif{z})\dif{t}$ be the compensator of $\{\F_{t}\}_{t \geq 0}$-adapted Poisson random measure $N(\dif{z},\dif{t})$ defined on $(\Omega \times \R^{+} \times (\R \backslash \{0\}),\F \times \mathcal{B}(\R^{+}) \times \mathcal{B}((\R \backslash \{0\})))$, where $\nu$ is a L\'{e}vy measure satisfying $\int_{\R \backslash \{0\}} \big(|y|^{2} \wedge 1\big) \nu(\dif{y}) < +\infty$. Moreover, $\{E(t)\}_{t \geq 0}$ is a subordinator and used to determine the random number of ``time steps" that occur within the subordinated process for a given unit of chronological time. Under the global Lipschitz conditions specified in Assumption \ref{As:(2.1)}, \cite[Lemma 4.1]{kobayashi2011stochastic} ensures that \eqref{eq:TCSDE} admits a unique solution $\{X(t)\}_{t \geq 0}$; see \cite{dinunno2024stochastic, meng2024ultimate, nane2017stability, nane2018path, ni2020study} for its more properties and applications.

In order to numerically approximate \eqref{eq:TCSDE}, the duality principle and the idea of approximations of inverse subordinator stimulate us to consider the stochastic $\theta$ method with $\theta \in [0,1]$ for the corresponding original SDEs of \eqref{eq:TCSDE} as follows  
\begin{align}\label{eq:Stmethod}
	  Y_{n+1}
	  =&~
	  Y_{n} + \theta f(t_{n+1},Y_{n+1})\Delta
	  +
	  (1-\theta) f(t_{n},Y_{n})\Delta
      +
      g(t_{n},Y_{n})\Delta W_{n}  \notag
      \\&~+
      \int_{t_{n}}^{t_{n+1}}\int_{|z|<c}
      h(t_n,Y_n,z) \,\tilde{N}(\dif{z},\dif{s}),
      \quad n = 0,1,2,\cdots
\end{align}
with $Y_0=Y(0)$, where $t_{n} := n\Delta$ and $\Delta W_{n} := W(t_{n+1}) - W(t_n)$ for $\Delta \in (0,1)$ and $n \in \N$. Different from the strong convergence analysis for stochastic $\theta$ method and its invariants on finite time interval, we establish the strong convergence order of stochastic $\theta$ method on the infinite time interval $\R^{+}$, which together with the approximation of the inverse subordinator gives the strong convergence order of numerical approximation for \eqref{eq:TCSDE}. For the case $\theta = 0$, we investigate the weak convergence order of Euler--Maruyama method on $\R^{+}$ with the help of the Kolmogorov backward partial integro differential equations and the higher order moment estimates of exact and numerical solutions. As a consequence of the approximation of the inverse subordinator, the weak  convergence order of numerical approximation for \eqref{eq:TCSDE} is derived.

The remainder of this paper is organized as follows. Section \ref{sect:asspre} provides some basic assumptions and preliminary results for \eqref{eq:TCSDE}. Then the strong convergence order of stochastic $\theta$ method and the weak convergence order of Euler--Maruyama method are derived in Section \ref{sect:strongorder} and Section \ref{sect:weakorder}, respectively. Some numerical
experiments are finally given in Section \ref{sect:numexp} to illustrate these theoretical results.

\section{Assumptions and preliminaries}\label{sect:asspre}
Throughout this paper, we use the following notations unless otherwise specified. For any $a,b \in \R$, we denote $a \wedge b := \min\{a, b\}$. Let $|\cdot|$ be the Euclidean norm in $\R^d$ and $\langle x,y \rangle $ denote the inner product of vectors $x,y \in \R^d$. By $A^{\top}$ we denote the transpose of a vector in $\R^{d}$ or matrix in $\R^{d \times m}$. Following the same notation as the vector norm, $|A| := \sqrt{\tr{(A^{\top}A)}}$ denotes the trace norm of a matrix $A \in \R^{d \times m}$. For simplicity, the letter $C$ denotes a generic positive constant that is independent of $\Delta \in (0,1), t \geq 0$ and varies for each appearance.

%
%
Let $\{D(t)\}_{t \geq 0}$ be a 
subordinator starting from $0$ with Laplace transform
\begin{equation}\label{eq:2.1}
	  \E\big[e^{-\lambda D(t)}\big]
	  =
	  e^{-t\phi(\lambda)},
	  \quad \lambda > 0, \ t \geq 0,
\end{equation}
where Laplace exponent $\phi(\lambda) := \int_{0}^{\infty}(1-e^{-\lambda x}) \,\nu(\dif{x})$. Besides, we always assume that $\{D(t)\}_{t \geq 0}$, $\{W(t)\}_{t \geq 0}$ and $\{N(t,\cdot)\}_{t \geq 0}$ are mutually independent. Define the inverse of $\{D(t)\}_{t \geq 0}$ as follows
$$E(t) := \inf\big\{s \geq 0 : D(s) > t \big\}, \quad t \geq 0,$$
Then $t \mapsto E(t)$ is continuous and nondecreasing almost surely. 
Moreover, \cite[Lemma 2.3]{deng2020semiimplicit} indicates that for any $\lambda > 0$ and $t \geq 0$, there exists $C > 0$, depending on $\lambda$ but independent of $t$, such that 
\begin{equation}\label{eq:estimateEelambdaEt}
      \E\big[e^{\lambda E(t)}\big] \leq Ce^{Ct}.  
\end{equation}

To provide efficient numerical approximation for \eqref{eq:TCSDE}, we make the following assumptions.
\begin{assumption}\label{As:(2.1)}
      There exists $C^{*} > 0$ such that for any $t \geq 0$ and $x,y \in \R^{d}$, 
	  \begin{equation*}
			|f(t,x)-f(t,y)|^{2} + |g(t,x)-g(t,y)|^{2}
			+
			\int_{|z|<c} |h(t,x,z) - h(t,y,z)|^{2} \,\nu(\dif{z})
			\leq
			C^{*}|x-y|^{2}.
	  \end{equation*}
      Besides, there exists $C > 0$ such that for any $t \geq 0$ and $x \in \R^{d}$,
      \begin{equation*}
			|f(t,x)|^{2} + |g(t,x)|^{2}
			+
			\int_{|z|<c} |h(t,x,z)|^{2} \,\nu(\dif{z})
			\leq
			C(1+|x|^{2}).
	  \end{equation*}
\end{assumption}

Let $\mathcal{G}_{t} := \F_{E{(t)}}$ for $t \geq 0$ and denote $\mathcal{L}(\mathcal{G}_{t})$ the class of left continuous with right limits and $\{\mathcal{G}_{t}\}_{t \geq 0}$-adapted processes. 
\begin{assumption}\label{As:(2.3)}
      If $\{X(t)\}_{t \geq 0}$ is a right continuous with left limits and $\mathcal{G}_t$-adapted process, then
      $$f(E(t),X(t)),g(E(t),X(t)),h(E(t),X(t),z)
        \in \mathcal{L}(\mathcal{G}_t).$$
\end{assumption}

It follows from Assumptions \ref{As:(2.1)}, \ref{As:(2.3)} and 
\cite[Lemma 4.1]{kobayashi2011stochastic} that \eqref{eq:TCSDE} admits a unique solution $\{X(t)\}_{t \in [0,T]}$, described by
\begin{align}\label{eq:TCSDEintegral}
	  X(t)
	  =&~
	  x_{0} + \int_{0}^{t} f(E(s),X(s-)) \diff{E(s)}
	  +
	  \int_{0}^{t} g(E(s),X(s-)) \diff{W(E(s))} \notag
	  \\&~+ 
	  \int_{0}^{t} \int_{|z|<c} 
	  h(E(s),X(s-),z) \,\tilde{N}(\dif{z},\dif{E(s)}), 
      \quad t \in [0,T],
\end{align}
which is right continuous with left limits, $\{\mathcal{G}_{t}\}_{t \geq 0}$-adapted. Moreover, for the corresponding original SDEs of \eqref{eq:TCSDE}
\begin{equation}\label{eq:SDE}
	  \diff{Y(t)}
	  =
	  f(t,Y(t-))\diff{t} + g(t, Y(t-))\diff{W(t)}
	  +
	  \int_{|z|<c}h(t,Y(t-),z)
	  \,\tilde{N}(\dif{z},\dif{t}),
	  \quad t > 0 
\end{equation}
with $Y(0) = x_{0}$, Assumption \ref{As:(2.1)} and \cite[Theorem 6.2.3]{applebaum2009levy} imply that \eqref{eq:SDE} possesses a unique solution $\{Y(t)\}_{t \geq 0}$, which is right continuous with left limits, $\{\F_t\}_{t \geq 0}$-adapted and given by
\begin{align}\label{eq:SDEintegral}
      Y(t)
	  =&~
	  x_{0} + \int_{0}^{t} f(s,Y(s-)) \diff{s}
	  +
      \int_{0}^{t} g(s, Y(s-)) \diff{W(s)} \notag
	  \\&~+
	  \int_{0}^{t} \int_{|z|<c} h(s,Y(s-),z)
      \,\tilde{N}(\dif{z},\dif{s}), \quad t \geq 0. 
\end{align} 
The following duality principle (see, e.g., \cite[Theorem 4.2]{kobayashi2011stochastic} and \cite[Section 3.3]{ni2020study}) provides a deep connection between the solution of the time-changed SDEs \eqref{eq:TCSDE} and that of the corresponding original SDEs \eqref{eq:SDE}, which will be frequently applied in the convergence analysis.

\begin{lemma}\label{le:2.5}
      Suppose Assumptions \ref{As:(2.1)} and \ref{As:(2.3)} hold. then we have the following conclusions:
	  \begin{itemize}
            \item [(1)] If $\{Y(t)\}_{t \geq 0}$ satisfies \eqref{eq:SDE}, then $\{Y(E(t))\}_{t \geq 0}$ satisfies \eqref{eq:TCSDE}.
				
            \item [(2)] If $\{X(t)\}_{t \geq 0}$ satisfies \eqref{eq:TCSDE}, then $\{X(D(t))\}_{t \geq 0}$ satisfies \eqref{eq:SDE}.
      \end{itemize}
\end{lemma}

For convenience, we will use $X(t),Y(t)$ to denote $X(t-),Y(t-)$ in \eqref{eq:TCSDEintegral} and \eqref{eq:SDEintegral}, respectively; see \cite[Remark 2.1]{dareiotis2016tamed}. The following two lemmas give the mean square moment estimate and the H\"{o}lder continuity for \eqref{eq:SDEintegral} on the infinite time interval $\R^{+}$.

\begin{lemma}\label{le:Ymoment}
      Suppose that Assumption \ref{As:(2.1)} holds. Then there exists $C > 0$, independent of $t$, such that 
      $$\E\big[|Y(t)|^{2}\big] \leq Ce^{Ct}, \quad t \geq 0.$$
\end{lemma}

\begin{proof}
      According to the It\^{o} formula (see, e.g., \cite[Lemma 2.6]{chen2019meansquare}), we have
      \begin{align*}
			|Y(t)|^{2}
			=&~
			|x_{0}|^2
			+
			2\int_{0}^{t} \big\langle Y(s),
			f(s,Y(s)) \big\rangle \diff{s}
			+
			2\int_{0}^{t} \big\langle Y(s),
			g(s,Y(s))\diff{W(s)} \big\rangle 
			\\&~+
			\int_{0}^{t} |g(s,Y(s))|^{2} \diff{s} 
			+
			\int_{0}^{t} \int_{|z|<c}
			\big(|Y(s) + h(s,Y(s),z)|^{2} - |Y(s)|^{2}\big)
			\,\tilde{N}(\dif{z},\dif{s})     
			\\&~+
			\int_{0}^{t}\int_{|z|<c}
			\big(|Y(s)+h(s,Y(s),z)|^{2} - |Y(s)|^{2} 
			- 2\big\langle Y(s),h(s,Y(s),z) \big\rangle\big) 
			\,\nu(\dif{z})\dif{s}.
      \end{align*}
	  Taking expectations and using 
	  Assumption \ref{As:(2.1)} lead to
	  \begin{align*}
			\E[|Y(t)|^{2}] 
			\leq&~
			|x_{0}|^{2}
			+
			\E\bigg[\int_0^t|Y(s)|^2\diff{s}\bigg]
			+
			\E\bigg[\int_0^t\int_{|z|<c}
			|h(s,Y(s),z)|^{2} \,\nu(\dif{z})\dif{s}\bigg]
			\\&~+
			\E\bigg[\int_0^t|f(s,Y(s)|^2\diff{s}\bigg]
			+
			\E\bigg[\int_0^t|g(s,Y(s)|^2\diff{s}\bigg]
			\\\leq&~ 
			|x_{0}|^{2} + Ct + C\int_{0}^{t} \E[|Y(s)|^{2}] \diff{s}.
      \end{align*}
      Then the Gronwall inequality (see, e.g., \cite[Theorem 1]{dragomir2003some}) completes the proof.
\end{proof}

\begin{lemma}\label{le:Yholdercontinuity}
      Suppose that Assumption \ref{As:(2.1)} holds. Then for any $t,s \geq 0$ with $0 \leq t-s < 1$, there exists $C > 0$, independent of $t$ and $s$, such that $$\E[|Y(t)-Y(s)|^{2}] \leq  (t-s)Ce^{Ct}.$$
\end{lemma}

\begin{proof}
	  Due to the It\^{o} isometry, the H\"{o}lder inequality and
	  $$Y(t)-Y(s) = \int_s^t f(r,Y(r))\diff{r}
		+
		\int_s^tg(r,Y(r))\diff{W(r)}
		+
		\int_s^t\int_{|z|<c}h(r,Y(r),z)
		\,\tilde{N}(\dif{z},\dif{r}),$$
      one has
	  \begin{align*}
			\E[|Y(t)-Y(s)|^{2}]
			\leq&~
			3\E\bigg[\Big|\int_{s}^{t} f(r,Y(r)) \diff{r}\Big|^{2}\bigg]
			+
			3\E\bigg[\int_s^t|g(r,Y(r))|^2\dif{r}\bigg]
			\\&~
			+
			3\E\bigg[\int_s^t\int_{|z|<c}
            |h(r,Y(r),z)|^{2} \,\nu(\dif{z})\dif{r}\bigg]
			\\\leq&~
			3(t-s)\int_s^t \E[|f(r,Y(r))|^2]\diff{r}
			+
			3\int_s^t \E[|g(r,Y(r))|^2]\diff{r}
			\\&~
			+
			3\int_s^t\E\bigg[\int_{|z|< c}
			|h(r,Y(r),z)|^{2} \nu(\dif{z})\bigg]\diff{r}.
      \end{align*}
      Utilizing $0 \leq t-s < 1$ and Assumption \ref{As:(2.1)}, 
      it holds that
      \begin{align*}
			\E[|Y(t)-Y(s)|^2]
			\leq&~
			C\int_s^t \big(1 + \E[|Y(r)|^2]\big) \diff{r}.
      \end{align*}
      Lemma \ref{le:Ymoment} further implies the required result.
\end{proof}

\section{Strong convergence order of stochastic theta method}
\label{sect:strongorder}
This section aims to present the strong convergence order of numerical approximations for \eqref{eq:TCSDE}. To this end, we discretize \eqref{eq:SDE} and $\{E(t)\}_{t \geq 0}$ by \eqref{eq:Stmethod} and \eqref{eq:3.8}, respectively. Then the composition of the numerical solution of \eqref{eq:SDE} and the discretized inverse subordinator is used to approximate the solution to the time-changed SDEs \eqref{eq:TCSDE}. The corresponding numerical approximation involves two types of errors: one is generated by the application of stochastic theta method and the other ascribed to simulation of the inverse subordinator.
Before analyzing these errors separately, 
%
%
%
%
%
%
let us first give the well-posedness of \eqref{eq:Stmethod} for $\theta \in (0,1]$.


\begin{lemma}
      Suppose that Assumption \ref{As:(2.1)} holds and let $\theta \sqrt{C^{*}} \Delta < 1$. Then the nonlinear equation given by stochastic theta method \eqref{eq:Stmethod} has a unique solution with probability one.
\end{lemma}

\begin{proof}
      For every step of the stochastic theta method \eqref{eq:Stmethod} with $Y_{n}$ given, to find $Y_{n+1}$ according to \eqref{eq:Stmethod} is equivalent to solve an implicit problem $F(x) = 0$, where
      \begin{align*}
			F(x)
			:=&~
			x - Y_n - \theta f(t_{n+1},x)\Delta
			-
			(1-\theta) f(t_n,Y_n)\Delta
			\\&~
			- g(t_n,Y_n)\Delta W_n 
			- \int_{t_{n}}^{t_{n+1}}\int_{|z|<c}
			h(t_n,Y_n,z) \,\tilde{N}(\dif{z},\dif{s}),
			\quad x \in \R.
      \end{align*}      
	  From Assumption \ref{As:(2.1)}, it follows that
	  \begin{align*}   
			\big\langle x-y,F(x)-F(y) \big\rangle
			=&~
			\big|x-y\big|^{2} - \theta \big\langle x-y,
			f(t_{n+1},x) - f(t_{n+1},y) \big\rangle\Delta
			\\\geq&~
			\big(1 - \theta\sqrt{C^{*}}\Delta\big)
			\big|x-y\big|^{2}.
      \end{align*}
      Then we complete the proof due to $\theta \sqrt{C^{*}} \Delta < 1$ and 
	  \cite[Theorem C.2]{stuart1996dynamical}.
\end{proof}

Define the piecewise continuous numerical solution by
\begin{equation}\label{eq:3.2}
	  Y_{\Delta}(t) = Y_{n}, 
      \quad t \in [t_{n},t_{n+1}), n = 0,1,2,\cdots.
\end{equation}
To obtain the error estimate between $\{Y(t)\}_{t \geq 0}$ and $\{Y_{\Delta}(t)\}_{t \geq 0}$, we assume the following H\"{o}lder continuity of coefficients with respect to the time variables.

\begin{assumption}\label{As:fghholdercont}
      There exist $C > 0$ and $\gamma > 0$ such that for any $t,s \geq 0$ and $x \in \R^{d}$, 
      \begin{align*}
			|f(t,x) &- f(s,x)|^{2} 
			+
			|g(t,x) - g(s,x)|^{2}
			\\&+
			\int_{|z|<c}|h(t,x,z) 
			- h(s,x,z)|^{2} \,\nu(\dif{z})
			\leq
			C(1 + |x|^2)|t-s|^{\gamma}.
      \end{align*}   
\end{assumption}

\begin{lemma}\label{Th:3.4}
      Suppose that Assumptions \ref{As:(2.1)} and \ref{As:fghholdercont} hold and let $\theta (1 + \sqrt{C^{*}}) \Delta < \rho$ with $\rho \in (0,\frac{1}{2})$. Then there exists $C > 0$ independent of $t$ and $\Delta$ such that
      $$\E\big[|Y(t)-Y_\Delta(t)|^{2}\big]
		\leq \Delta^{\gamma \wedge 1} Ce^{Ct}, 
		\quad t \geq 0.$$
\end{lemma}

\begin{proof}
      For any $t \in [t_{n},t_{n+1}), n \in \N$, we have
      \begin{equation}\label{eq:YtmYdt2}
			\E\big[|Y(t)-Y_{\Delta}(t)|^{2}\big]
			=
			\E\big[|Y(t)-Y_{n}|^{2}\big]
			\leq
			2\E\big[|Y(t)-Y(t_{n})|^{2}\big]
			+
			2\E\big[|Y(t_{n})-Y_{n}|^{2}\big].
      \end{equation}
      Lemma \ref{le:Yholdercontinuity} implies that
      \begin{align}\label{eq:YtYtn2}
			\E\big[|Y(t) - Y(t_{n})|^{2}\big]
			\leq 
			(t-t_{n})Ce^{Ct}
			\leq 
			\Delta Ce^{Ct}.
      \end{align}
      Based on \eqref{eq:SDEintegral} and \eqref{eq:Stmethod}, we utilize
      \begin{align*}
			Y(t_{n})-Y_{n}
			=&~
			Y(t_{n-1}) - Y_{n-1}
			+
			\theta\int_{t_{n-1}}^{t_{n}}
			f(s,Y(s)) - f(t_{n},Y_{n}) \diff{s}
			\\&~+
			(1-\theta)\int_{t_{n-1}}^{t_{n}}
			f(s,Y(s)) - f(t_{n-1},Y_{n-1}) \diff{s}
			\\&~+
			\int_{t_{n-1}}^{t_{n}}
			g(s,Y(s)) - g(t_{n-1},Y_{n-1}) \diff{W(s)}
			\\&~+
			\int_{t_{n-1}}^{t_{n}} \int_{|z|<c} 
			h(s,Y(s),z) - h(t_{n-1},Y_{n-1},z)
			\,\tilde{N}(\dif{z},\dif{s})
      \end{align*}
      to obtain
      \begin{align}\label{eq:In1In2result}
			\E\big[|Y(t_{n})-Y_{n}|^{2}\big]
			=&~
			\E\bigg[\Big\langle Y(t_{n}) - Y_{n},
			\theta\int_{t_{n-1}}^{t_{n}}
			f(s,Y(s)) - f(t_{n},Y_{n}) \diff{s}
			\Big\rangle\bigg] \notag
			\\&~+
			\E\bigg[\Big\langle Y(t_{n}) - Y_{n},
			Y(t_{n-1}) - Y_{n-1}  \notag
			\\&~+
			(1-\theta)\int_{t_{n-1}}^{t_{n}}
			f(s,Y(s)) - f(t_{n-1},Y_{n-1}) \diff{s} \notag
			\\&~+
			\int_{t_{n-1}}^{t_{n}} \int_{|z|<c} 
			h(s,Y(s),z) - h(t_{n-1},Y_{n-1},z)
			\,\tilde{N}(\dif{z},\dif{s})  \notag
			\\&~+
			\int_{t_{n-1}}^{t_{n}}
			g(s,Y(s)) - g(t_{n-1},Y_{n-1}) \diff{W(s)}
			\Big\rangle\bigg] \notag
			\\=:&~
			I_{n}^{1} + I_{n}^{2}.
      \end{align}
      Note that
      \begin{align*}
			I_{n}^{1}
			=&~
			\theta\E\bigg[\int_{t_{n-1}}^{t_{n}}
			\big\langle Y(t_{n}) - Y_{n},
			f(s,Y(s)) - f(s,Y(t_{n})) \big\rangle
			\diff{s} \bigg]
			\\&~+
			\theta\E\bigg[\int_{t_{n-1}}^{t_{n}}
			\big\langle Y(t_{n}) - Y_{n},
			f(s,Y(t_{n})) - f(t_{n},Y(t_{n})) \big\rangle
			\diff{s} \bigg]
			\\&~+
			\theta\E\bigg[\int_{t_{n-1}}^{t_{n}}
			\big\langle Y(t_{n}) - Y_{n},
			f(t_{n},Y(t_{n})) - f(t_{n},Y_{n}) \big\rangle
			\diff{s} \bigg]
			\\=:&~ 
            I_{n}^{11} + I_{n}^{12} + I_{n}^{13}.
	  \end{align*}
      By Assumption \ref{As:(2.1)} and Lemma \ref{le:Yholdercontinuity}, one gets
      \begin{align*}
			I_{n}^{11}
			\leq&~
			\frac{\theta}{2}\E\bigg[\int_{t_{n-1}}^{t_{n}}
			|Y(t_{n}) - Y_{n}|^{2}
			+
			|f(s,Y(s)) - f(s,Y(t_{n}))|^{2}
			\diff{s} \bigg]
			\\\leq&~
			\frac{\theta}{2}\Delta
			\E\big[|Y(t_{n}) - Y_{n}|^{2}\big]
			+
			\frac{\theta}{2}C^{*} \int_{t_{n-1}}^{t_{n}}
			\E\big[|Y(s) - Y(t_{n})|^{2}\big]\diff{s}
			\\\leq&~
			\frac{\theta}{2}\Delta
			\E\big[|Y(t_{n}) - Y_{n}|^{2}\big]
			+
			\frac{\theta}{2}\Delta^{2}C^{*}C e^{Ct_{n}}.
	  \end{align*}
      Similarly, we use Assumption \ref{As:fghholdercont} and Lemma \ref{le:Ymoment} to derive that 
	  \begin{align*}
			I_{n}^{12}
			\leq&~
			\frac{\theta}{2}\Delta
			\E\big[|Y(t_{n}) - Y_{n}|^{2}\big]
			+
			\frac{\theta}{2}\E\bigg[\int_{t_{n-1}}^{t_{n}}
			|f(s,Y(t_{n})) - f(t_{n},Y(t_{n}))|^{2}
			\diff{s} \bigg]
			\\\leq&~
			\frac{\theta}{2}\Delta
			\E\big[|Y(t_{n}) - Y_{n}|^{2}\big]
			+
			\frac{\theta}{2}\Delta^{1+\gamma}
			C\big(1 + \E\big[|Y(t_{n})|^{2}\big]\big)
			\\\leq&~
			\frac{\theta}{2}\Delta
			\E\big[|Y(t_{n}) - Y_{n}|^{2}\big]
			+
			\frac{\theta}{2}\Delta^{1+\gamma}
			C\big(1 + Ce^{Ct_{n}}\big).
      \end{align*}
      Making use of Assumption \ref{As:(2.1)} yields
	  \begin{align*}
			I_{n}^{13}
			\leq&~
			\theta\E\bigg[\int_{t_{n-1}}^{t_{n}}
			|Y(t_{n}) - Y_{n}|
			|f(t_{n},Y(t_{n})) - f(t_{n},Y_{n})| 
			\diff{s} \bigg]
            \\\leq&~
			\theta\sqrt{C^{*}}\Delta
			\E\big[|Y(t_{n}) - Y_{n}|^{2}\big].
      \end{align*}
      It follows that
      \begin{align}\label{eq:In1result}
			I_{n}^{1}
			\leq&~
			\theta\big(\sqrt{C^{*}} + 1\big)\Delta
			\E\big[|Y(t_{n}) - Y_{n}|^{2}\big]
			+
			\frac{\theta}{2}\Delta^{2}C^{*}C e^{Ct_{n}}
			+
			\frac{\theta}{2}\Delta^{1+\gamma}
			C\big(1 + Ce^{Ct_{n}}\big).
	  \end{align}
      Concerning $I_{n}^{2}$, the independence between $\{W(t)\}_{t \geq 0}$ and $\{\tilde{N}(t,\cdot)\}_{t \geq 0}$ as well as the $\F_{t_{n}}$-measurability of $Y(t_{n})$ and $Y_{n}$ enable us to get 
      \begin{align*}
			I_{n}^{2}
			\leq&~
			\frac{1}{2}\E\big[|Y(t_{n}) - Y_{n}|^{2}\big]
			+
			\frac{1}{2}
			\E\big[|Y(t_{n-1}) - Y_{n-1}|^{2}\big]
			\\&~+
			\frac{1}{2}(1-\theta)^{2}
			\E\bigg[\Big|\int_{t_{n-1}}^{t_{n}}
			f(s,Y(s)) - f(t_{n-1},Y_{n-1}) 
			\diff{s}\Big|^{2}\bigg]
			\\&~+
			\frac{1}{2}\E\bigg[\Big|
			\int_{t_{n-1}}^{t_{n}} \int_{|z|<c} 
			h(s,Y(s),z) - h(t_{n-1},Y_{n-1},z)
			\,\tilde{N}(\dif{z},\dif{s})\Big|^{2}\bigg]
			\\&~+
			\frac{1}{2}\E\bigg[\Big|\int_{t_{n-1}}^{t_{n}}
			g(s,Y(s)) - g(t_{n-1},Y_{n-1}) 
			\diff{W(s)}\Big|^{2}\bigg]
			\\&~+
			(1-\theta)
			\E\bigg[\Big\langle Y(t_{n-1}) - Y_{n-1},
			\int_{t_{n-1}}^{t_{n}}
			f(s,Y(s)) - f(t_{n-1},Y_{n-1}) 
			\diff{s}\Big\rangle\bigg]
			\\&~+
			(1-\theta)
			\E\bigg[\Big\langle \int_{t_{n-1}}^{t_{n}}
			f(s,Y(s)) - f(t_{n-1},Y_{n-1}) \diff{s},
            \\&~\int_{t_{n-1}}^{t_{n}} \int_{|z|<c} 
			h(s,Y(s),z) - h(t_{n-1},Y_{n-1},z)
			\,\tilde{N}(\dif{z},\dif{s})
            \Big\rangle\bigg]
			\\&~+
			(1-\theta)
			\E\bigg[\Big\langle \int_{t_{n-1}}^{t_{n}}
			f(s,Y(s)) - f(t_{n-1},Y_{n-1}) \diff{s},
            \\&~\int_{t_{n-1}}^{t_{n}}
			g(s,Y(s)) - g(t_{n-1},Y_{n-1}) \diff{W(s)}
            \Big\rangle\bigg].
      \end{align*}
      Then the H\"{o}lder inequality and the It\^{o} isometry imply that
	  \begin{align*}
			I_{n}^{2}            
			\leq&~
			\frac{1}{2}\E\big[|Y(t_{n}) - Y_{n}|^{2}\big]
			+
			\frac{1 + (1-\theta)\Delta}{2}
			\E\big[|Y(t_{n-1}) - Y_{n-1}|^{2}\big]
			\\&~+
			2\int_{t_{n-1}}^{t_{n}}\E\big[|
			f(s,Y(s)) - f(t_{n-1},Y_{n-1})|^{2}\big]\diff{s}
			\\&~+
			\int_{t_{n-1}}^{t_{n}} \int_{|z|<c} 
			\E\big[|h(s,Y(s),z) - h(t_{n-1},Y_{n-1},z)|^{2}\big]
			\,\nu(\dif{z})\dif{s}
			\\&~+
			\int_{t_{n-1}}^{t_{n}}
			\E\big[|g(s,Y(s)) - g(t_{n-1},Y_{n-1})|^{2}\big]
			\diff{s}
			\\=:&~
			\frac{1}{2}\E\big[|Y(t_{n}) - Y_{n}|^{2}\big]
			+
			\frac{1 + (1-\theta)\Delta}{2}
			\E\big[|Y(t_{n-1}) - Y_{n-1}|^{2}\big]
			\\&~+
			I_{n}^{21} + I_{n}^{22} + I_{n}^{23}.
      \end{align*}
      Because of Assumptions \ref{As:(2.1)} and \ref{As:fghholdercont}, one has
	  \begin{align*}
			I_{n}^{21}
			\leq&~
			6\int_{t_{n-1}}^{t_{n}}
			\E\big[|f(s,Y(s)) 
			- f(s,Y(t_{n-1}))|^{2}\big]\diff{s}
			\\&~+
			6\int_{t_{n-1}}^{t_{n}}
			\E\big[|f(s,Y(t_{n-1})) 
			- f(t_{n-1},Y(t_{n-1}))|^{2}\big]
			\diff{s}
			\\&~+
			6\int_{t_{n-1}}^{t_{n}}
			\E\big[|f(t_{n-1},Y(t_{n-1})) 
			- f(t_{n-1},Y_{n-1})|^{2}\big]\diff{s}
			\\\leq&~
			6C\Delta^{1+\gamma}
			\big(1 + \E\big[|Y(t_{n-1})|^{2}\big]\big)
			+
			6C^{*}\Delta
			\E\big[|Y(t_{n-1}) - Y_{n-1}|^{2}\big]
			\\&~+
			6C^{*}\int_{t_{n-1}}^{t_{n}}
			\E\big[|Y(s) - Y(t_{n-1})|^{2}\big]\diff{s}.
      \end{align*}
      In the same way, we obtain
      \begin{align*}
			I_{n}^{22}
			\leq&~
			3C\Delta^{1+\gamma}
			\big(1 + \E\big[|Y(t_{n-1})|^{2}\big]\big)
			+
			3C^{*}\Delta
			\E\big[|Y(t_{n-1}) - Y_{n-1}|^{2}\big] 
			\\&~+
			3C^{*}\int_{t_{n-1}}^{t_{n}}
			\E\big[|Y(s) - Y(t_{n-1})|^{2}\big]\diff{s} 
      \end{align*}
	  and
	  \begin{align*}
			I_{n}^{23}
			\leq&~
			3C\Delta^{1+\gamma}
			\big(1 + \E\big[|Y(t_{n-1})|^{2}\big]\big)
			+
			3C^{*}\Delta
			\E\big[|Y(t_{n-1}) - Y_{n-1}|^{2}\big]
			\\&~+
		    3C^{*}\int_{t_{n-1}}^{t_{n}}
			\E\big[|Y(s) - Y(t_{n-1})|^{2}\big]\diff{s}.          
	  \end{align*}
      In combination with Lemmas \ref{le:Ymoment} and \ref{le:Yholdercontinuity}, we deduce that
      \begin{align}\label{eq:In2result} 
			I_{n}^{2} 
			\leq&~
			\frac{1}{2}\E\big[|Y(t_{n}) - Y_{n}|^{2}\big]
			+
			\bigg(\frac{1 + (1-\theta)\Delta}{2} 
			+ 12C^{*}\Delta\bigg)
			\E\big[|Y(t_{n-1}) - Y_{n-1}|^{2}\big] \notag
			\\&~+
			12C^{*}\int_{t_{n-1}}^{t_{n}}
			\E\big[|Y(s) - Y(t_{n-1})|^{2}\big]\diff{s}
			+
			12C\Delta^{1+\gamma}
			\big(1 + \E\big[|Y(t_{n-1})|^{2}\big]\big) \notag
			\\\leq&~
			\frac{1}{2}\E\big[|Y(t_{n}) - Y_{n}|^{2}\big]
			+
			12C^{*}\Delta^{2}Ce^{Ct_{n}} \notag
			+
			12C\Delta^{1+\gamma}\big(1 + Ce^{Ct_{n-1}}\big)
			\\&~+
			\bigg(\frac{1 + (1-\theta)\Delta}{2} 
			+ 12C^{*}\Delta\bigg)
			\E\big[|Y(t_{n-1}) - Y_{n-1}|^{2}\big]. 
	  \end{align}
      Inserting \eqref{eq:In1result} and \eqref{eq:In2result} into \eqref{eq:In1In2result} yields
	  \begin{align*}
			&~\bigg(\frac{1}{2} 
			- \theta(1 + \sqrt{C^{*}})\Delta\bigg)
			\E\big[|Y(t_{n})-Y_{n}|^{2}\big]
			\\\leq&~
			\bigg(\frac{1 + (1-\theta)\Delta}{2} 
			+ 12C^{*}\Delta\bigg)
			\E\big[|Y(t_{n-1}) - Y_{n-1}|^{2}\big] \notag
			+
			\Delta^{2 \wedge (1+\gamma)}C e^{Ct_{n}}.           
	  \end{align*} 
      and thus 
      \begin{align*}
			\E\big[|Y(t_{n})-Y_{n}|^{2}\big]
			\leq
			\frac{\frac{1 + (1-\theta)\Delta}{2} 
				+ 12C^{*}\Delta}{\frac{1}{2} 
				- \theta(1 + \sqrt{C^{*}})\Delta}
			\E\big[|Y(t_{n-1}) - Y_{n-1}|^{2}\big] \notag
			+
			\frac{\Delta^{2 \wedge (1+\gamma)}C e^{Ct_{n}}}
			{\frac{1}{2} - \theta(1 + \sqrt{C^{*}})\Delta}.           
      \end{align*}
      Noting that $\big(\frac{1}{2} - \theta (1 + \sqrt{C^{*}})\Delta\big)^{-1} \leq \big(\frac{1}{2} - \rho\big)^{-1}$ and
      \begin{align*}
			&~\E\big[|Y(t_{n})-Y_{n}|^{2}\big]
			+
			\sum_{i=1}^{n-1}
			\E\big[|Y(t_{i})-Y_{i}|^{2}\big]
			=
			\sum_{i=1}^{n}
			\E\big[|Y(t_{i})-Y_{i}|^{2}\big]
			\\\leq&~
			\bigg(1 + \frac{\frac{1-\theta}{2} + 12C^{*}
				+ \theta(1 + \sqrt{C^{*}})}{\frac{1}{2} 
				- \theta(1 + \sqrt{C^{*}})\Delta}\Delta\bigg)
			\sum_{i=1}^{n-1}\E\big[|Y(t_{i}) - Y_{i}|^{2}\big] \notag
			+
			\frac{n\Delta^{2 \wedge (1+\gamma)}C e^{Ct_{n}}}
			{\frac{1}{2} - \theta(1 + \sqrt{C^{*}})\Delta},         
      \end{align*}
      we obtain
      \begin{align*}
			\E\big[|Y(t_{n})-Y_{n}|^{2}\big]
			\leq&~
			\frac{\frac{1-\theta}{2} + 12C^{*}
				+ \theta(1 + \sqrt{C^{*}})}{\frac{1}{2} - \rho}
			\Delta\sum_{i=1}^{n-1}
			\E\big[|Y(t_{i}) - Y_{i}|^{2}\big]  \notag
			+
			\frac{n\Delta^{2 \wedge (1+\gamma)}C e^{Ct_{n}}}
			{\frac{1}{2} - \rho}.
      \end{align*}
      The discrete Gronwall inequality \cite[Lemma 3.4]{mao2013strong} implies that
      \begin{align}\label{eq:YtnmYn2}
			\E\big[|Y(t_{n})-Y_{n}|^{2}\big]
			\leq
			\Delta^{1 \wedge \gamma}C e^{Ct_{n}}
			\leq
			\Delta^{1 \wedge \gamma}C e^{Ct}
			e^{C(t_{n}-t)}
			\leq
			\Delta^{1 \wedge \gamma}C e^{Ct}       
      \end{align}
      due to $t_{n}-t \leq \Delta < 1$ and $n\Delta = t_n \leq \frac{1+Ct_{n}}{C} \leq \frac{e^{Ct_n}}{C}$ for all $n \in \N$ and some $C > 0$. Plugging \eqref{eq:YtmYdt2} and \eqref{eq:YtnmYn2} into \eqref{eq:YtYtn2} finishes the proof.
\end{proof}

Now we turn to approximate $\{E(t)\}_{t \in [0,T]}$ via $\{E_{\Delta}(t)\}_{t \in [0,T]}$, which is defined by
\begin{equation}\label{eq:3.8}
      E_{\Delta}(t)
      :=
      (\min\{n \in \N : D(t_n)>t\} - 1)\Delta,
      \quad t \in [0,T].
\end{equation}
Then each sample path of stochastic process $\{E_{\Delta}(t)\}_{t \in [0,T]}$ is a non-decreasing step function with constant jump size $\Delta$, and the procedure is stopped when $T \in [D(t_{N}), D(t_{N+1}))$ for some $N \in \N$. It follows that $E_\Delta(T) = t_{N}$ and 
\begin{equation}\label{eq:EDeltattn}
      E_{\Delta}(t)
	  =
	  t_{n}, \quad t \in [D(t_{n}),D(t_{n+1})),
	  n = 0, 1,\cdots, N.
\end{equation}
Moreover, $\{E(t)\}_{t \in [0,T]}$ can be well approximated by $\{E_{\Delta}(t)\}_{t \in [0,T]}$ in the sense that 
\begin{equation}\label{eq:EtEDeltaterror}
	  E(t)-\Delta 
      \leq
	  E_{\Delta}(t)\leq E(t), 
      \quad t \in [0,T].
\end{equation}

%
%
%
%

Lemma \ref{le:2.5} indicates that the solution $\{X(t)\}_{t \in [0,T]}$ to \eqref{eq:TCSDE} can be expressed as $Y(E(t))$ with $Y(t)$ being the solution to  \eqref{eq:SDE}. Thus, it is reasonable to approximate $\{X(t)\}_{t \in [0,T]}$ by the stochastic process $\{X_{\Delta}(t)\}_{t \in [0,T]}$, which is described by
\begin{equation}\label{eq:3.9}
	  X_{\Delta}(t)
	  :=
	  Y_{\Delta}(E_\Delta(t)),
	  \quad t \in [0,T],
\end{equation}
where $Y_{\Delta}$ and $E_\Delta $ are defined by \eqref{eq:3.2} and \eqref{eq:3.8}, respectively. Based on the above preparations, we are in a position to establish the convergence order of numerical approximations given by \eqref{eq:3.9}.

\begin{theorem}\label{th:strongorder}
      Suppose that Assumptions \ref{As:(2.1)}, \ref{As:(2.3)} and \ref{As:fghholdercont} hold and let $\theta \sqrt{C^{*}} \Delta < 1$. Then there exists $C > 0$, independent of $T$ and $\Delta$, such that
      \begin{equation*}
			\E\big[|X(T)-X_\Delta(T)|^{2}\big]
			\leq
			Ce^{CT}\Delta^{\gamma \wedge 1}.
      \end{equation*}
\end{theorem}

\begin{proof}
      Owing to \eqref{eq:3.9} and Lemma \ref{le:2.5}, one has
      \begin{align*}
			&~\E\big[|X(T)-X_{\Delta}(T)|^{2}\big]
			= 
			\E\big[|Y(E(T))-Y_{\Delta}(E_{\Delta}(T))|^{2}\big]
			\\\leq&~
			2\E\big[|Y(E(T))-Y(E_{\Delta}(T))|^{2}\big]
			+
			2\E\big[|Y(E_{\Delta}(T))
            -
            Y_{\Delta}(E_{\Delta}(T))|^{2}\big].
      \end{align*}
      From \eqref{eq:EtEDeltaterror}, we get $0 \leq E(T)-E_{\Delta}(T)\leq \Delta$, which together with \eqref{eq:estimateEelambdaEt} and Lemma \ref{le:Yholdercontinuity} yields
      \begin{equation*}
			\E\big[|Y(E(T))-Y(E_\Delta(T))|^{2}\big]
			\leq
			\Delta C \E\big[e^{CE(T)}\big]
			\leq
			C\Delta e^{CT}.
      \end{equation*}
      Applying \eqref{eq:estimateEelambdaEt}, \eqref{eq:EtEDeltaterror} and Lemma \ref{Th:3.4} leads to
      \begin{equation*}
			\E\big[|Y(E_\Delta(T))
			- Y_\Delta(E_\Delta(T))|^{2}\big]
			\leq
			C\Delta^{\gamma\wedge1}
			\E\big[e^{CE_{\Delta}(T)}\big]
			\leq
			C\Delta^{\gamma\wedge1}
			\E\big[e^{CE(T)}\big]
			\leq
			C\Delta^{\gamma\wedge1}e^{CT}.
      \end{equation*}
      It follows that
      \begin{equation*}
			\E\big[|X(T)-X_\Delta(T)|^{2}\big]
			\leq
			C\Delta e^{CT}
			+
			C\Delta^{\gamma\wedge1}e^{CT}
			\leq
			C\Delta^{\gamma\wedge1}e^{CT},
      \end{equation*}
      where $C > 0$ is independent of $T$ and $\Delta$. 
\end{proof}

\section{Weak convergence order of Euler--Maruyama method}
\label{sect:weakorder}
This section focuses on the weak convergence order of numerical approximations for \eqref{eq:TCSDE}. Similar to the treatment of strong convergence analysis in Section \ref{sect:strongorder}, the corresponding weak approximation still involves two types of errors: one is due to the simulation of inverse subordinator and the other ascribed to the following Euler--Maruyama method (i.e., \eqref{eq:Stmethod} with $\theta = 0$)
\begin{align}\label{eq:4.1}
	  Y_{n+1}
	  =
	  Y_n + f(t_n,Y_n)\Delta + g(t_n,Y_n)\Delta{W_{n}}  
	  +
	  \int_{t_{n}}^{t_{n+1}}
	  \int_{|z|<c}h(t_n,Y_n,z)
	  \,\tilde{N}(\dif{z},\dif{t}), \quad n \in \N
\end{align}
with $Y_0 = x_{0}$, where $t_n := n\Delta$ and $\Delta W_n := W(t_{n+1})-W(t_n)$. Defining a continuous-time version of $\{Y_{n}\}_{n \in \N}$ by
\begin{align}\label{eq:4.2}
	  Y_{\Delta}(t)
	  =&~
	  Y_0
	  +
	  \int_{0}^{t}
	  f(\check{s},Y_{\Delta}(\check{s}))\diff{s}
	  +
	  \int_{0}^{t}
	  g(\check{s},Y_{\Delta}(\check{s}))
	  \diff{W(s)}\notag
	  \\&~+
	  \int_{0}^{t}\int_{|z|<c}
	  h(\check{s},Y_{\Delta}(\check{s}),z)
	  \,\tilde{N}(\dif{z},\dif{s}), 
	  \quad t \geq 0
\end{align}      
with $\check{s} := \big(\max\{n \in \N : n\Delta < s\}\big)\Delta$, we have $\check{t} = n\Delta = t_{n}$ for $t \in [t_n,t_{n+1})$ and $Y_{\Delta}(t_{n}) = Y_{n}, n \in \N$. By imposing a stronger version of the linear growth condition on the coefficients with respect to the state variables, one can derive the higher order moment of $\{Y(t)\}_{t \geq 0}$ and $\{Y_{\Delta}(t)\}_{t \geq 0}$.

\begin{assumption}\label{As:fghgrowthcond}
      For any $p \geq 2$, there exists a positive constant $K_{fgh} \geq 1$ such that for all $t \geq 0$ and $x \in \R^{d}$,  
	  \begin{equation*}
		    |f(t,x)|^{p} + |g(t,x)|^{p}
		    + 
            \int_{|z|<c}|h(t,x,z)|^{p} \,\nu(\dif{z})
		    \leq
		    K_{fgh}(1+|x|^{p}).
	  \end{equation*}
\end{assumption}

\begin{lemma}\label{lm:Ytpexactbound}
      Suppose that Assumptions \ref{As:(2.1)} and \ref{As:fghgrowthcond} hold. Then for any $p \geq 2$, there exists $C > 0$, independent of $t$, such that
      \begin{equation*}
			\E\big[|Y(t)|^{p}\big]
			\leq 
			Ce^{Ct}, \quad t \geq 0.  
      \end{equation*}
\end{lemma}

\begin{proof}
      Applying the It\^{o} formula indicates that for any $t \geq 0$,
	  \begin{align*}
			|Y(t)|^{p}
			=&~
			|Y(0)|^{p}
			+
			p\int_{0}^{t}|Y(s)|^{p-2} 
			\big\langle Y(s),f(s,Y(s)) 
			\big\rangle\diff{s}
			\\&~+
			\frac{p}{2}\int_{0}^{t}
			|Y(s)|^{p-2}|g(s,Y(s))|^{2}\diff{s}
			\\&~+
			\frac{p(p-2)}{2}\int_0^t|Y(s)|^{p-4}
			|Y(s)^{\top}g(s,Y(s))|^{2}\diff{s}
			\\&~+
			p\int_0^t |Y(s)|^{p-2}|\big\langle Y(s),
			g(s,Y(s)) \diff{W(s)}\big\rangle
			\\&+
			\int_{0}^{t}\int_{|z|<c} \big(|Y(s) 
			+ h(s,Y(s),z)|^{p} - |Y(s)|^{p}\big)
			\tilde{N}(\dif{z},\dif{s})     
			\\&+
			\int_{0}^{t}\int_{|z|<c}
			\big(|Y(s)+h(s,Y(s),z)|^{p} -|Y(s)|^{p} 
			\\&~- p|Y(s)|^{p-2}\big\langle Y(s),
			h(s,Y(s),z) \big\rangle\big)
			\,\nu(\dif{z})\dif{s}.
	  \end{align*}
      By taking expectations, we have
      \begin{align*}
			\E\big[|Y(t)|^{p}\big] 
			\leq&~
			|Y(0)|^{p}
			+
			p\E\bigg[\int_{0}^{t} 
			|Y(s)|^{p-1}|f(s,Y(s))|\diff{s}\bigg]
			\\&~+
			\frac{p(p-1)}{2}\E\bigg[\int_0^t
			|Y(s)|^{p-2}|g(s,Y(s))|^2\diff{s}\bigg]  
			\\&~+
			\E\bigg[\int_{0}^{t}\int_{|z|<c}|Y(s)
            + h(s,Y(s),z)|^{p} \,\nu(\dif{z})\dif{s}\bigg]
			\\&~+
			p\E\bigg[\int_0^t\int_{|z|<c}|Y(s)|^{p-1} 
            |h(s,Y(s),z)|\,\nu(\dif{z})\dif{s}\bigg]
			\\&~+
			\nu(|z|<c)\E\bigg[\int_{0}^{t}
			|Y(s)|^{p} \diff{s}\bigg].
      \end{align*}
      From the Young inequality and Assumption \ref{As:fghgrowthcond}, it follows that
      \begin{align*}
			\E\big[|Y(t)|^{p}\big] 
			\leq&~
			|Y(0)|^{p}
			+
			\E\bigg[\int_{0}^{t} (p-1)|Y(s)|^{p} 
			+ |f(s,Y(s))|^{p}\diff{s}\bigg]
			\\&~+
			\frac{p-1}{2}\E\bigg[\int_0^t
			(p-2)|Y(s)|^{p}
			+ 2|g(s,Y(s))|^{p}\diff{s}\bigg]  
			\\&~+
			2^{p}\E\bigg[\int_0^t\int_{|z|<c}
			\big(|Y(s)|^{p} + |h(s,Y(s),z)|^{p}\big) 
			\,\nu(\dif{z})\dif{s}\bigg]
			\\&~+
			\E\bigg[\int_0^t\int_{|z|<c}(p-1)|Y(s)|^{p} 
			+ |h(s,Y(s),z)|^{p}\,\nu(\dif{z})\dif{s}\bigg]
			\\&~+
			\nu(|z|<c)\E\bigg[\int_0^t|Y(s)|^p\dif{s}\bigg]
			\\\leq&~
			|Y(0)|^{p} + Ct + C\int_0^t \E [|Y(s)|^{p}]\diff{s}
			\\\leq&~
			Ce^{Ct} + C\int_{0}^{t} \E\big[|Y(s)|^{p}\big]\diff{s}.
      \end{align*}
      Owing to the Gronwall inequality, 
      we complete the proof.
\end{proof}


\begin{lemma}\label{lemma 2.4}
      Suppose that Assumptions \ref{As:(2.1)} and \ref{As:fghgrowthcond} hold. Then for any $p \geq 2$, there exists $C > 0$, independent of $t > 0$, such that
      \begin{equation*}
            \E\big[|Y_{\Delta}(t)|^{p}\big]
			\leq Ce^{Ct}, \quad t \geq 0.
	  \end{equation*}
\end{lemma}

\begin{proof}
      Applying the It\^{o} formula shows that for any $t \geq 0$,
      \begin{align*}
		    |Y_{\Delta}(t)|^{p}
			=&~
			|Y(0)|^{p}
			+
			p\int_{0}^{t} |Y_\Delta(s)|^{p-2}
			\big\langle Y_\Delta(s), 
		    f(\check{s},Y_{\Delta}(\check{s}))
			\big\rangle\diff{s}
		    \\&~+
			\frac{p}{2}\int_0^t |Y_\Delta(s)|^{p-2}
			|g(\check{s},Y_{\Delta}(\check{s}))|^2 \diff{s}
			\\&~+
		    \frac{p(p-2)}{2}\int_{0}^{t}
			|Y_\Delta(s)|^{p-4}|Y_\Delta(s)^{\top}
		    g(\check{s},Y_{\Delta}(\check{s}))|^2 \diff{s}
			\\&~+
			p\int_{0}^{t} |Y_{\Delta}(s)|^{p-2}
		    \big\langle Y_\Delta(s),
			g(\check{s},Y_{\Delta}(\check{s}))
			\diff{W(s)}\big\rangle
			\\&~+
			\int_{0}^{t}\int_{|z|<c}\big(|Y_\Delta(s) 
			+ h(\check{s},Y_{\Delta}(\check{s}),z)|^{p}
			- |Y_{\Delta}(s)|^{p}\big)\,\tilde{N}(\dif{z},\dif{s})     
			\\&~+
			\int_0^t\int_{|z|<c} \big(|Y_\Delta(s) 
			+ h(\check{s},Y_{\Delta}(\check{s}),z)|^p-|Y_\Delta(s)|^p
			\\&~- p|Y_\Delta(s)|^{p-2}\big\langle Y_\Delta(s), 
            h(\check{s},Y_{\Delta}(\check{s}),z) \big\rangle\big)\,\nu(\dif{z})\dif{s}.
      \end{align*} 
      By taking expectations and using the Young inequality, we have
      \begin{align*}
	        \E\big[|Y_{\Delta}(t)|^{p}\big] 
			\leq&~
			|Y(0)|^{p}
			+
			p\E\bigg[\int_{0}^{t} |Y_{\Delta}(s)|^{p-1}
			|f(\check{s},Y_{\Delta}(\check{s}))|\diff{s}\bigg]
			\\&~+
			\frac{p(p-1)}{2} \E\bigg[\int_{0}^{t}
			|Y_{\Delta}(s)|^{p-2}
		    |g(\check{s},Y_{\Delta}(\check{s}))|^{2}
			\diff{s}\bigg]  
			\\&~+
			\E\bigg[\int_{0}^{t}\int_{|z|<c} |Y_{\Delta}(s) 
			+ h(\check{s},Y_{\Delta}(\check{s}),z)|^{p} \,\nu(\dif{z})\dif{s}\bigg]
		    \\&~+
			p\E\bigg[\int_0^t\int_{|z|<c}|Y_\Delta(s)|^{p-1} 
            |h(\check{s},Y_{\Delta}(\check{s}),z)|
			\,\nu(\dif{z})\dif{s}\bigg]
			\\\leq&~
			|Y(0)|^{p}
			+
			\E\bigg[\int_{0}^{t} (p-1)|Y_{\Delta}(s)|^{p}
			+
			|f(\check{s},Y_{\Delta}(\check{s}))|^{p}\diff{s}\bigg]
			\\&~+
			\frac{p-1}{2} \E\bigg[\int_{0}^{t}
			(p-2)|Y_{\Delta}(s)|^{p}
			+
			2|g(\check{s},Y_{\Delta}(\check{s}))|^{p}
			\diff{s}\bigg]  
			\\&~+
			2^{p-1}\E\bigg[\int_{0}^{t}\int_{|z|<c} 
			|Y_{\Delta}(s)|^{p} 
			+ 
			|h(\check{s},Y_{\Delta}(\check{s}),z)|^{p} \,\nu(\dif{z})\dif{s}\bigg]
			\\&~+
		    \E\bigg[\int_{0}^{t}\int_{|z|<c}
			(p-1)|Y_{\Delta}(s)|^{p}
			+
			|h(\check{s},Y_{\Delta}(\check{s}),z)|^{p}
			\,\nu(\dif{z})\dif{s}\bigg].
      \end{align*}
	  As a consequence of Assumption \ref{As:fghgrowthcond}, we obtain
	  \begin{align*}
			\E\big[|Y_{\Delta}(t)|^{p}\big] 
			\leq&~
			|Y(0)|^{p}
			+
			\bigg(\frac{p(p-1)}{2}
			+
			(2^{p-1} + p-1)\nu(|z|<c)\bigg) \int_{0}^{t}
			\E\big[|Y_{\Delta}(s)|^{p}\big]\diff{s}
			\\&~+
			\int_{0}^{t}\E\big[
			|f(\check{s},Y_{\Delta}(\check{s}))|^{p}
			\big]\diff{s}
			+
			(p-1)\int_{0}^{t} \E\big[
			|g(\check{s},Y_{\Delta}(\check{s}))|^{p}
			\big]\diff{s}  
			\\&~+
			\big(2^{p-1} + 1\big)\E\bigg[
			\int_{0}^{t}\int_{|z|<c} 
			|h(\check{s},Y_{\Delta}(\check{s}),z)|^{p} \,\nu(\dif{z})\dif{s}\bigg]
			\\\leq&~
			|Y(0)|^{p}
			+
			\bigg(\frac{p(p-1)}{2}
			+
			(2^{p-1} + p-1)\nu(|z|<c)\bigg) \int_{0}^{t}
			\E\big[|Y_{\Delta}(s)|^{p}\big]\diff{s}
			\\&~+
			C\int_{0}^{t}\big(1 + \E\big[
			|Y_{\Delta}(\check{s})|^{p}
			\big]\big)\diff{s}
			+
			(p-1)C\int_{0}^{t}\big(1 + \E\big[
			|Y_{\Delta}(\check{s})|^{p}
			\big]\big)\diff{s} 
			\\&~+
			\big(2^{p-1} + 1\big)C
			\int_{0}^{t}\big(1 + \E\big[
			|Y_{\Delta}(\check{s})|^{p}
			\big]\big)\diff{s}
			\\\leq&~
			|Y(0)|^{p} + Ct
			+
			C\int_{0}^{t}
			\E\big[|Y_{\Delta}(s)|^{p}\big]\diff{s}
			+
			C\int_{0}^{t}\E\big[
			|Y_{\Delta}(\check{s})|^{p}
			\big]\diff{s}.
      \end{align*}
      It follows that
      \begin{align*}
			\sup_{r \in [0,t]}
			\E\big[|Y_{\Delta}(r)|^{p}\big] 
			\leq&~
			Ce^{Ct}
			+
			C\int_{0}^{t} \sup_{r \in [0,s]}
			\E\big[|Y_{\Delta}(r)|^{p}\big]\diff{s}.
      \end{align*}
      By the Gronwall inequality, we obtain the required result and complete the proof.
\end{proof}

The following lemma estimates the weak error of $\{Y(t)\}_{t \geq 0}$ with respect to different times. For this purpose, let $C_{p}^{k}(\R^{d};\R)$ be the set of $k \in \{1,2,\cdots\}$ times continuously differentiable functions from $\R^{d}$ to $\R$ which, together with their partial derivatives of order up to $k$, have at most polynomial growth.
\begin{lemma}\label{lemma 2.5}
      Suppose that Assumptions \ref{As:(2.1)} and \ref{As:fghgrowthcond} hold and let $\Phi \in C_{p}^{2}(\R^{d};\R)$. Then for any $s,t \geq 0$ with $0 \leq t-s \leq \Delta$, there exists $C > 0$, independent of $\Delta$ and $t$, such that
      \begin{align*}
			\big|\E\big[\Phi(Y(t)) - \Phi(Y(s))\big]\big| 
			\leq 
			\Delta Ce^{Ct}, \quad t \geq 0.
      \end{align*}
\end{lemma}

\begin{proof}
      By the It\^{o} formula, we have
      \begin{align*}
			&~\Phi(Y(t)) - \Phi(Y(s))
			\\=&~
			\int_{s}^{t} \Phi^{'}(Y(r))f(r,Y(r))
			+
			\frac{1}{2}  \tr\big(g(r,Y(r))^{\top}
			\Phi^{''}\big(Y(r)\big)
			g(r,Y(r))\big) \diff{r}
			\\&~+
			\int_{s}^{t} \Phi^{'}(Y(r))g(r,Y(r)) \diff{W(r)}
			+
			\int_{s}^{t} \int_{|z|<c}
			\Phi\big(Y(r) + h(r,Y(r),z)\big) 
			- \Phi\big(Y(r)\big)
			\,\tilde{N}(\dif{z},\dif{r})
			\\&~+
			\int_{s}^{t}\int_{|z|<c}
			\Phi\big(Y(r) + h(r,Y(r),z)\big) 
			- \Phi\big(Y(r)\big)
			- \Phi^{'}\big(Y(r)\big)h(r,Y(r),z)) 
			\,\nu(\dif{z})\dif{r}.
      \end{align*}
      Taking expectations and utilizing mean-value formula lead to
      \begin{align*}
			\big|\E\big[\Phi(Y(t)) - \Phi(Y(s))\big]\big|
			\leq&~
			\E\bigg[\int_{s}^{t} 
			\big|\Phi^{'}(Y(r))\big|
			\big|f(r,Y(r))\big|
			\\&~+
			\frac{1}{2} \big|\tr\big(g(r,Y(r))^{\top}
			\Phi^{''}\big(Y(r)\big)
			g(r,Y(r))\big)\big| \diff{r}\bigg]
			\\&~+
			\E\bigg[\int_{s}^{t}\int_{|z|<c}
			\big|\Phi^{'}\big(Y(r) 
            + \kappa h(r,Y(r),z)\big)\big| 
			\big| h(r,Y(r),z))\big|
			\\&~+
			\big|\Phi^{'}\big(Y(r)\big)h(r,Y(r),z))\big| 
			\,\nu(\dif{z})\dif{r}\bigg]
      \end{align*} 
      for $\kappa \in (0,1)$. Owing to Assumption \ref{As:(2.1)} and $\Phi \in C_{p}^{2}(\R^{d};\R)$, there exists a polynomial $\Psi \colon \R^{+} \to \R^{+}$ such that
      \begin{align*}
			\big|\E\big[\Phi(Y(t)) - \Phi(Y(s))\big]\big|
			\leq&~
			C\E\bigg[\int_{s}^{t} \Psi(|Y(r)|)
			\Big(\big|f(r,Y(r))\big|
			+
			\big|g(r,Y(r))\big|^{2}
			\\&~+ 
			\int_{|z|<c}
			\big| h(r,Y(r),z))\big|
			\,\nu(\dif{z}) \Big)\dif{r}\bigg]
			\\\leq&~
			C\int_{s}^{t} \E\big[\Psi(|Y(r)|)
			\big(1 + |Y(r)|^{2}\big)\big]\diff{r}
			\\\leq&~
			C\int_{s}^{t} 
			\big(\E\big[\Psi(|Y(r)|)^{2}\big]
			\big)^{\frac{1}{2}}
			\big(1 + \E\big[|Y(r)|^{4}\big]
			\big)^{\frac{1}{2}}\diff{r}
			\\\leq&~
			\int_{s}^{t} Ce^{Cr} \diff{r}
			\leq
			\Delta Ce^{Ct},
      \end{align*}
      where the H\"{o}lder inequality and Lemma \ref{lm:Ytpexactbound}, have been used.
\end{proof}

With the above preparations, we are in a position to reveal the weak error estimate of Euler--Maruyama method on the infinite time interval $\R^{+}$. 
\begin{lemma}\label{lm:2.6}
      Assume that all components of $f(t,\cdot), g(t,\cdot), h(t,\cdot,z)$ belong to $C_{p}^{4}(\R^{d};\R)$ for all $t \geq 0$ and $z \in \R \backslash \{0\}$. Suppose that Assumptions \ref{As:(2.1)} and \ref{As:fghgrowthcond} hold and let $\Phi \in C_{p}^{4}(\R^{d};\R)$. Then there exists $C > 0$, independent of $\Delta$ and $t$, such that
      \begin{equation*}
			\big|\E\big[\Phi(Y(t))
			- \Phi(Y_\Delta(t))\big]\big| 
			\leq 
			\Delta Ce^{Ct}, \quad t \geq 0. 
	  \end{equation*}
\end{lemma}

\begin{proof}
      Fix $t > 0$ and define $u(s,x) := \E\big[\Phi(Y(t)) \,|\, Y(s) = x\big]$ for $s \in [0,t]$ and $x \in \R^{d}$. According to \cite[Lemma 12.3.1]{platen2010numerical}, $u(s,x)$ is the unique solution of the Kolmogorov backward partial integro differential equations
      \begin{align}\label{eq:Kbpide}
            &~\partial_{s}u(s,x) + \Big((\partial_{x}u(s,x))f(s,x)
			+
			\frac{1}{2}\operatorname{trace}\big(g^{\top}(s,
            x)\partial_{x}^{2}u(s,x)g(s,x)\big)\Big)\diff{t}
			\\&~+
			\int_{|z|<c}u\big(s,x+h(s,x,z)\big) - u(s,x)
			- \partial_{x}u(s,x)h(s,x,z) \,\nu(\dif{z})
			=
			0 \notag
	  \end{align}
      for all $(s,x) \in (0,t) \times \R^{d}$ with the terminal condition $u(t,x) = \Phi(x), x \in \R^{d}$. Moreover, the function $x \mapsto u(s,x)$ belongs to $C_{p}^{4}(\R^{d};\R)$ for each $s \in [0,t]$ (see \cite[Lemma 12.3.1]{platen2010numerical}). Together with the It\^{o} formula, one has 
      \begin{align*}
			\diff{u(s,Y(s))}
			=&~
			(\partial_{x}u(s,Y(s))) g(s,Y(s)) \diff{W(s)}
			\\&~+ 
			\int_{|z|<c}u\big(s,Y(s)+h(s,Y(s),z)\big)
			- u(s,Y(s))\,\tilde{N}(\dif{z},\dif{s}),
            \quad s \in (0,T].
      \end{align*}
      It follows that $\E\big[u(0,Y(0))\big] = \E\big[u(t,Y(t))\big] = \E\big[\Phi(Y(t))\big]$, which implies
      \begin{align*}
			\big|\E\big[\Phi(Y(t))
			- \Phi(Y_{\Delta}(t))\big]\big|
			=
			\big|\E\big[u(t,Y_{\Delta}(t)) 
			- u(t,Y(t))\big]\big|
			=
			\big|\E\big[u(t,Y_{\Delta}(t))
			- u(0,Y_{\Delta}(0))\big]\big|.
      \end{align*}
      Based on the It\^{o} formula and \eqref{eq:Kbpide}, we obtain
      \begin{align*}
            &~\big|\E\big[u(t,Y_{\Delta}(t))
			- u(0,Y_{\Delta}(0))\big]\big|
            \\=&~
			\bigg|\E\bigg[\int_{0}^{t}
		    \partial_{s}u(s,Y_{\Delta}(s))
			+
			\big(\partial_{x}u(s,Y_{\Delta}(s))\big)
			f(\check{s},Y_{\Delta}(\check{s}))
			\\&~+
			\frac{1}{2}\operatorname{trace}
            \big(g^{\top}(\check{s},Y_{\Delta}(\check{s}))
			(\partial_{x}^{2}u(s,Y_\Delta(s)))
            g(\check{s},Y_{\Delta}(\check{s}))\big)
			\\&~+
			\int_{|z| < c} u\big(s,Y_{\Delta}(s) 
            + h(\check{s},Y_{\Delta}(\check{s}),z)\big)
			- u(s,Y_{\Delta}(s))
			\\&~-
			\big(\partial_{x}u(s,Y_{\Delta}(s))\big) 
            h(\check{s},Y_{\Delta}(\check{s}),z)
			\,\nu(\dif{z})\dif{s}\bigg]\bigg|
			\\\leq&~
			I_{1} + I_{2} + I_{3} + I_{4}
      \end{align*}
      with
      \begin{align*}
			&I_{1}
			:=
			\int_{0}^{t} \big|\E\big[
			\big(\partial_{x}u(s,Y_{\Delta}(s))\big)\big(f(s,Y_{\Delta}(s))
			- f(\check{s},Y_{\Delta}(\check{s}))\big)\big]\big|\diff{s},
			\\&
			I_{2}
			:=
	        \int_{0}^{t}\big|\E\big[\big(\partial_{x}^{2} 
            u(s,Y_{\Delta}(s))\big)\big(|g(s,Y_\Delta(s))|^{2}
		    - |g(\tau_s,Y_\Delta(\tau_s))|^{2}\big) \big]\big|\diff{s},
			\\&
			I_{3}
			:=
			\int_{0}^{t} \bigg|\E\bigg[
			\partial_{x}u(s,Y_\Delta(s))\bigg(u(s,Y_\Delta(s))
			\int_{|z|<c} h(s,Y_{\Delta}(s),z)
			- h(\check{s},Y_{\Delta}(\check{s}),z)
			\,\nu(\dif{z})\bigg)\bigg]\bigg| \diff{s},
			\\&
			I_{4}
			:=
			\int_{0}^{t} \bigg|\E\bigg[ \int_{|z|<c}
			u\big(s,Y_{\Delta}(s) 
			+ h(\check{s},Y_{\Delta}(\check{s}),z)\big)
			- u\big(s,Y_{\Delta}(s) + h(s,Y_{\Delta}(s),z)\big) 
			\,\nu(\dif{z})\bigg]\bigg|\diff{s}.
      \end{align*}
      Let us estimate $I_{1}$, $I_{2}$, $I_{3}$ and $I_{4}$ one by one. we split $I_{1}$ as follows
      \begin{align*}
			I_{1}
			\leq&~
			\int_{0}^{t} 
			\big|\E\big[\big(\partial_{x}u(s,Y_{\Delta}(s))\big)
			f(s,Y_{\Delta}(s)))          
			-
			\big(\partial_{x}u(\check{s},Y_{\Delta}(\check{s}))\big)
			f(\check{s},Y_{\Delta}(\check{s})) \big]\big|\diff{s}
			\\&~+
			\int_{0}^{t}
			\big|\E\big[\big(\partial_{x}u(s,Y_{\Delta}(s))
			-
			\partial_{x}u(\check{s},Y_{\Delta}(\check{s}))\big)
            f(\check{s},Y_{\Delta}(\check{s}))\big]\big|\diff{s}
			\\=:&~
			I_{11} + I_{12}.
      \end{align*}
      Letting $\varphi(s,x) := (\partial_{x}u(s,x))f(s,x), s \in (0,t), x \in \R^{d}$  
	  and using the It\^{o} formula enable us to get 
      \begin{align*}
			I_{11}
			\leq&~
			\int_{0}^{t}
			\int_{\check{s}}^{s}\E\bigg[
			\Big|\partial_{r}\varphi(r,Y_\Delta(r))
			+
			\big(\partial_{x}\varphi(r,Y_\Delta(r))\big)
			f(\check{r},Y_\Delta(\check{r}))
			\\&~+
			\frac{1}{2}\big(\partial_{x}^{2}
            \varphi(r,Y_{\Delta}(r))\big)|
			g(\check{r},Y_\Delta(\check{r}))|^2
			\\&~+
			\int_{|z|<c}\varphi\big(r,Y_\Delta(r)
			+
			h(\check{r},Y_\Delta(\check{r}),z)\big)
            - \varphi(r,Y_\Delta(r))
			\\&~~~~
            - \big(\partial_{x}\varphi(r,Y_{\Delta}(r)\big)
            h(\check{r},Y_{\Delta}(\check{r}),z)
            \,\nu(\dif{z})\Big|\bigg]\diff{r}\dif{s}
	        \\\leq&~
			\int_0^t\int_{\check{s}}^{s}\E\bigg[
			|\partial_{r}\varphi(r,Y_\Delta(r))|
			+
			|\partial_{x}\varphi(r,Y_\Delta(r))|^{2}
			\\&~+
			|f(\check{r},Y_\Delta(\check{r}))|^{2}
			+
            |\partial_{x}^2\varphi(r,Y_\Delta(r))|^{2}
			+
            |g(\check{r},Y_\Delta(\check{r}))|^{4}
			\\&~
            +
            \int_{|z|<c}|\varphi\big(r,Y_\Delta(r)
            +
            h(\check{r},Y_\Delta(\check{r}),z)\big)|
            +
            |\varphi(r,Y_\Delta(r))|
			\\&~
            +
            |\partial_{x}\varphi(r,Y_\Delta(r)|^{2}
			+
			|h(\check{r},Y_\Delta(\check{r}),z)|^{2}
			\,\nu(\dif{z})\bigg]\diff{r}\dif{s}.
	  \end{align*}
      In view of Assumption \ref{As:fghgrowthcond}, there exist polynomials $P_{1},P_{2},P_{3}$ and $P_{4},P_{5}$ such that 
      \begin{align*}
			I_{11}
			\leq&~
			\int_{0}^{t}\int_{\check{s}}^{s}
            \E\bigg[P_{1}(|Y_\Delta(r)|)
            +
            P_{2}(|Y_\Delta(\check{r})|)
			\\&~+ 
            \int_{|z|<c}P_{3}(|Y_\Delta(r)
			+
			h(\check{r},Y_\Delta(\check{r}),z)|)
            \,\nu(\dif{z})\bigg]\diff{r}\dif{s}
			\\\leq&~ 
			\int_{0}^{t}\int_{\check{s}}^{s}
			\E\big[P_{4}(|Y_\Delta(r)|)
            +
            P_{5}(|Y_\Delta(\check{r})|)\big]\diff{r}\dif{s},
      \end{align*}
      where the elementary inequality $|\sum_{i=1}^{n}u_i|^p \leq n^{p-1}\sum_{i=1}^{n} |u_{i}|^{p}, p \geq 1, u_{i} \in \R, i = 1,2, \cdots,n$ has been used. Then Lemma \ref{lemma 2.4} implies 
      \begin{align*}
			I_{11}
			\leq
			\int_{0}^{t}\int_{\check{s}}^{s} Ce^{Cr} \diff{r}\dif{s}
			\leq
			\int_{0}^{t} \Delta Ce^{Cs} \diff{s}
			\leq
			\Delta Ce^{Ct}.
      \end{align*} 
      Similarly, we obtain $I_{12} \leq \Delta Ce^{Ct}$, and consequently
      \begin{align*}
			I_{1}
			\leq
			I_{11}+I_{12}
			\leq
			\Delta Ce^{Ct}.
      \end{align*}
      In the same way, one can derive that
      \begin{align*}
			I_{2} \leq \Delta Ce^{Ct},
            \quad
			I_{3} \leq \Delta Ce^{Ct}.
	  \end{align*}
      Concerning $I_4$, we have 
	  \begin{align*}
			I_{4}
			\leq&~
			\int_{0}^{t}\bigg|\E\bigg[\int_{|z|<c}
			u\big(s,Y_{\Delta}(s)
			+
			h(\check{s},Y_\Delta(\check{s}),z)\big)
			\\&~-
		    u\big(\check{s},Y_\Delta(\check{s})
		    +   
			h(\check{s},Y_\Delta(\check{s}),z)\big)
            \,\nu(\dif{z})\bigg]\bigg|\diff{s}
			\\&~~+
			\int_{0}^{t}\bigg|\E\bigg[\int_{|z|<c}
			u\big(s,Y_{\Delta}(s)
			+
			h(s,Y_{\Delta}(s),z)\big)
			\\&~-
			u\big(\check{s},Y_{\Delta}(\check{s})
			+
            h(\check{s},Y_{\Delta}(\check{s}),z)\big)
            \,\nu(\dif{z})\bigg]\bigg|\diff{s}
			\\=:&~
			I_{41}+I_{42}.
	  \end{align*}
      Denote $$\phi_{n}(s,x) := \int_{|z|<c} u\big(s,x+h(t_n,Y_n,z)\big) \,\nu(\dif{z}), \quad s \in [t_n,t_{n+1}], x \in \R^{d}.$$ It follows from the It\^{o} formula that for any $s \in [t_n,t_{n+1}]$, \begin{align*}
            &~\E\big[\phi_{n}(s,Y_\Delta(s))
            - \phi_{n}(t_n,Y_\Delta(t_n))\big]
			\\\leq&~
			\int_{t_n}^s\E\bigg[\Big|\partial_{r}
		    \phi_{n}(r,Y_\Delta(r))
		    +
		    \big(\partial_{x}\phi_{n}(r,Y_\Delta(r))\big)
			f(\check{r},Y_\Delta(\check{r}))
			\\&~+
			\frac{1}{2}\big(\partial_{x}^{2}
            \phi_{n}(r,Y_\Delta(r))\big)
            |g(\check{r},Y_\Delta(\check{r}))|^{2}
			\\&~+
			\int_{|z|<c}\phi_{n}\big(r,Y_\Delta(r)
			+ 
			h(\check{r},Y_{\Delta}(\check{r}),z)\big)
            - \phi_{n}(r,Y_{\Delta}(r))
			\\&~-
			\big(\partial_{x}\phi_{n}(r,Y_\Delta(r))\big)
			h(\check{r},Y_\Delta(\check{r}),z)
			\,\nu(\dif{z})\Big|\bigg]\diff{r}.
      \end{align*}
      Similar to the arguments for $I_{11}$, one can utilize Assumption \ref{As:fghgrowthcond} and Lemma \ref{lemma 2.4} to derive that 
      \begin{align*}
			\E\big[\phi_{n}(s,Y_\Delta(s)) - \phi_{n}(t_n,Y_\Delta(t_n))\big]
			\leq
			\int_{t_n}^s Ce^{Cr}\diff{r}
			\leq
			\Delta Ce^{Cs}, \quad s \in [t_n,t_{n+1}],
      \end{align*}
      which implies
      \begin{align*} 
            I_{41}
            =&~
            \sum_{n=0}^{\frac{\check{t}}{\Delta}-1}
            \int_{t_{n}}^{t_{n+1}}
            \bigg|\E\bigg[\int_{|z|<c} u\big(s,Y_\Delta(s)
            + h(t_{n},Y_\Delta(t_{n}),z)\big)
            \\&~- 
            u\big(t_{n},Y_\Delta(t_{n}) 
            + h(t_{n},Y_\Delta(t_{n}),z)\big)
            \,\nu(\dif{z})\bigg]\bigg|\diff{s}
			\\&~+
			\int_{\check{t}}^{t}\bigg|\E\bigg[\int_{|z|<c}
            u\big(s,Y_\Delta(s) 
            + h(\check{t},Y_\Delta(\check{t}),z)\big)
			\\&~-
			u\big(\check{t},Y_\Delta(\check{t})
			+
            h(\check{t},Y_\Delta(\check{t}),z)\big)
            \,\nu(\dif{z})\bigg]\bigg|\diff{s}
            \\=&~
            \sum_{n=0}^{\frac{\check{t}}{\Delta}-1}
            \int_{t_{n}}^{t_{n+1}}
            \big|\E\big[\phi_{n}(s,Y_\Delta(s))
            -\phi_{n}(t_n,Y_\Delta(t_n))\big]\big|\diff{s}
            \\&~+
            \int_{\check{t}}^{t}
            \Big|\E\Big[\phi_{\frac{\check{t}}{\Delta}}(s,Y_\Delta(s))
            -\phi_{\frac{\check{t}}{\Delta}}
            (\check{t},Y_\Delta(\check{t}))\Big]\Big|\diff{s}
			\\\leq&~
            \sum_{n=0}^{\frac{\check{t}}{\Delta}-1}
            \int_{t_{n}}^{t_{n+1}} \Delta Ce^{Cs} \diff{s}
            +
            \int_{\check{t}}^{t} \Delta Ce^{Cs} \diff{s}
			\\\leq&~
			\Delta Ce^{Ct}.
      \end{align*}
      By defining $$\psi(s,x) := \int_{|z|<c} u\big(s,x+ h(s,x,z)\big) \,\nu(\dif{z}), \quad s \in [0,T], x \in \R^{d},$$ we have
      \begin{align*} 
            I_{42}
			=
			\int_{0}^{t}\big|\E\big[\psi(s,Y_\Delta(s))
            - \psi(\check{s},Y_{\Delta}(\check{s})\big]\big|\diff{s},
      \end{align*}
      Similar to processing $I_{11}$, it holds that $I_{42}\leq \Delta Ce^{Ct}$, which accordingly gives $I_4\leq I_{41}+I_{42}\leq\Delta Ce^{Ct}$.
      As a consequence of the previous estimations, we finally obtain
	  \begin{align*}
			|\E[\Phi(Y(t))-\Phi(Y_\Delta(t))]|
			\leq I_{1} + I_{2} + I_{3} + I_{4}
			\leq \Delta Ce^{Ct},
      \end{align*}
      where $C > 0$ is independent of $\Delta$ and $t$. 
\end{proof}


Similar to \eqref{eq:3.9}, we approximate $\{X(t)\}_{t \in [0,T]}$ by $\{X_{\Delta}(t)\}_{t \in [0,T]}$ as follows
\begin{equation}\label{eq:5.3}
      X_{\Delta}(t)
	  :=
      Y_{\Delta}(E_\Delta(t)), \quad t \in [0,T],
\end{equation}
where $\{Y_{\Delta}(t)\}_{t \geq 0}$ and $\{E_{\Delta}(t)\}_{t \geq 0}$ are given by \eqref{eq:4.2} and \eqref{eq:3.8}, respectively. Then the weak convergence order of $\{X_{\Delta}(t)\}_{t \in [0,T]}$ is demonstrated by the following theorem.

\begin{theorem}\label{th:weakorder}
      Assume that all components of $f(t,\cdot), g(t,\cdot), h(t,\cdot,z)$ belong to $C_{p}^{4}(\R^{d};\R)$ for all $t \geq 0$ and $z \in \R \backslash \{0\}$. Suppose that Assumptions \ref{As:(2.1)}, \ref{As:(2.3)} and \ref{As:fghgrowthcond} hold and let $\Phi \in C_{p}^{4}(\R^{d})$. Then there exists $C > 0$, independent of $\Delta$, such that
      \begin{align*}
			\big|\E\big[\Phi(X(T))-\Phi(X_\Delta(T))\big]\big|
			\leq
			C\Delta.    
      \end{align*}
\end{theorem}

\begin{proof}
      By \eqref{eq:5.3} and Lemma \ref{le:2.5}, we have
      \begin{align*}
			&~\big|\E\big[\Phi(X(T))
		    - \Phi(X_\Delta(T))\big]\big|
			=
			\big|\E\big[\Phi(Y(E(T))
			- \Phi(Y_\Delta(E_\Delta (T))\big]\big|
			\\\leq&~
			\big|\E\big[\Phi(Y(E(T)) - \Phi(Y(E_\Delta (T))\big]\big|
			+
			\big|\E\big[\Phi(Y(E_\Delta (T)) - \Phi(Y_\Delta(E_\Delta (T))\big]\big|.    
      \end{align*}
      Applying \eqref{eq:estimateEelambdaEt}, \eqref{eq:EtEDeltaterror} and Lemma \ref{lemma 2.5} leads to
      \begin{align*}
			\big|\E\big[\Phi(Y(E(T)) - \Phi(Y(E_\Delta (T))\big]\big|
			\leq
			\Delta C\E\big[e^{CE(T)}\big]
			\leq 
			\Delta Ce^{CT}.
      \end{align*}
      Owing to \eqref{eq:estimateEelambdaEt}, \eqref{eq:EtEDeltaterror} and Lemma \ref{lm:2.6}, one gets
      \begin{align*}
			\big|\E\big[\Phi(Y(E_{\Delta}(T))
			- \Phi(Y_{\Delta}(E_{\Delta}(T))\big]\big|
			\leq
			\Delta C\E\big[e^{CE_{\Delta}(T)}\big]
			\leq
			\Delta C\E\big[e^{CE(T)}\big]
			\leq 
			\Delta Ce^{CT}.
	  \end{align*}
	  It follows that the desired result holds.
\end{proof}

\section{Numerical experiments}
\label{sect:numexp}

This section provides some numerical experiments to illustrate the previous theoretical results. Consider the one-dimensional SDE
\begin{align}\label{eq:example}
      \diff{X(t)}
      =&~
	  \big(\sin(E(t)) + X(t-)\big)\diff{E(t)}
	  +
	  \big(E(t) + \sin(X(t-))\big)\diff{W(E(t))} \notag 
	  \\&~+
	  \int_{\R}X(t-)z
	  \,\tilde{N}(\dif{z},\dif{E(t)}),
	  \quad t\in(0,T]
\end{align}
with $X(0) = 1$, where $\{W(t)\}_{t \geq 0}$ is a $\R$-valued standard Brownian motion on the complete filtered probability space $(\Omega,\F,\P,\{\F_{t}\}_{t \geq 0})$. Let $\tilde{N}(\dif{z},\dif{t}) := N(\dif{z},\dif{t}) - \nu(\dif{z})\dif{t}$ be the compensator of $\{\F_{t}\}_{t \geq 0}$-adapted Poisson random measure $N(\dif{z},\dif{t})$, where $\nu$ is a L\'{e}vy measure with $\nu(\dif{z}) = \frac{2}{\sqrt{2\pi}} e^{-\frac{z^{2}}{2}}\dif{z}, z \in \R$\cite{wang2016numerical}. Moreover, $E(t) := \inf\big\{s \geq 0 : D(s) > t \big\}, t \geq 0,$
be the inverse of the $\alpha$-stable subordinator. Noting that \eqref{eq:example} satisfies Assumptions \ref{As:(2.1)}, \ref{As:fghholdercont} and \ref{As:fghgrowthcond}, Theorems \ref{th:strongorder} and \ref{th:weakorder} thus hold.

\begin{figure}[!htbp]
      \subfigure[$\alpha = 0.45$]
	  {\includegraphics[width=0.45\textwidth]
	  {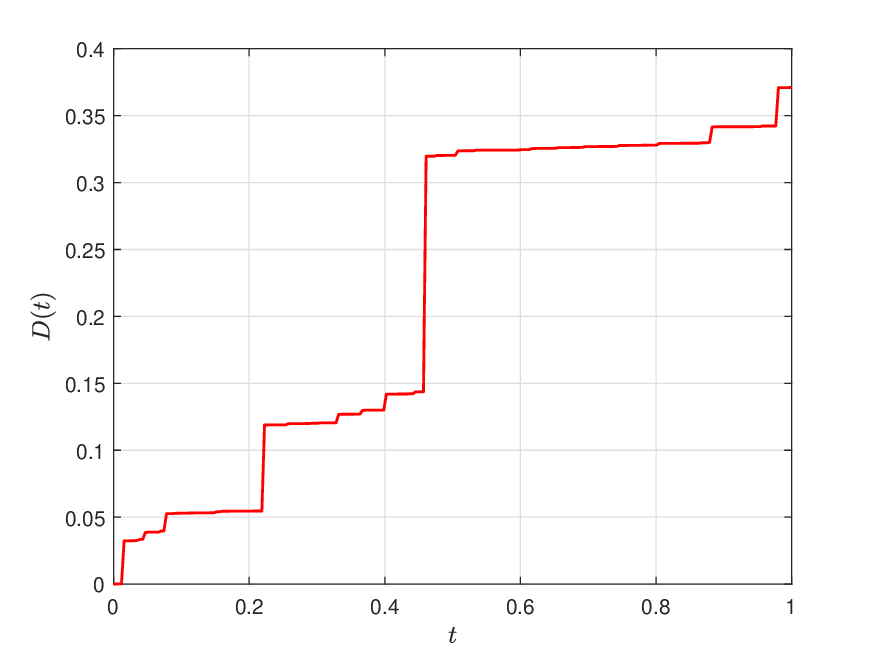}}
      \subfigure[$\alpha = 0.90$]
      {\includegraphics[width=0.45\textwidth]
      {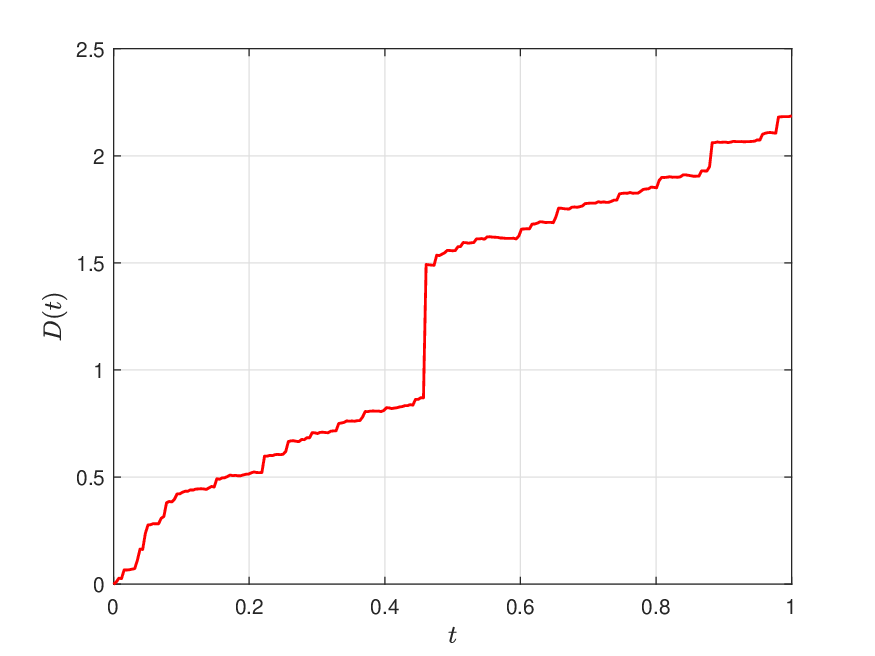}}
      \centering\caption {Sample paths for 
      $\{D(t)\}_{t \in [0,1]}$ with 
      $\Delta = 2^{-8}$}
      \label{fig:OneDpI1}
\end{figure}

Given $\Delta \in (0,1)$, the increments of $\alpha$--stable subordinator $\{D(t)\}_{t \geq 0}$ on each time interval $[n\Delta,(n+1)\Delta], n \in \N$ are independent random variables with the common distribution $S(\alpha,1,\Delta^{\frac{1}{\alpha}},0)$, where $S(\alpha,\beta,\gamma,\delta)$ denotes the stable distribution with parameters $(\alpha,\beta,\gamma,\delta)$ \cite{cont2004financial}. With the help of the Matlab command ``makedist" and ``random", we plot the smaple paths for $\{D(t)\}_{t \in [0,1]}$ with $\alpha = 0.45, 0.90$ and $\Delta = 2^{-8}$ in Figure \ref{fig:OneDpI1}. In view of \eqref{eq:3.8} and \eqref{eq:EDeltattn}, we follow the idea presented in \cite{jum2016strong} to find $N \in \N$ such that $T \in [D(N\Delta),D((N+1)\Delta))$. Based on
\begin{equation}\label{eq:525252}
	  E_{\Delta}(t)
	  =
	  \left\{\begin{aligned}
	  & n\Delta,  &t \in [D(n\Delta),D((n+1)\Delta)),
		n = 0,1,2,\cdots,N-1,
	  \\
	  & N\Delta,  &t \in [D(N\Delta),T],
      \end{aligned}\right.
\end{equation}
the sample paths of the inverse subordinator $\{E_{\Delta}(t)\}_{t \in [0,1]}$ is plotted in Figure \ref{fig:OneDpII} with $\Delta = 2^{-9}$.

\begin{figure}[!htbp]
      \subfigure[$\alpha = 0.45$]
      {\includegraphics[width=0.45\textwidth]
      {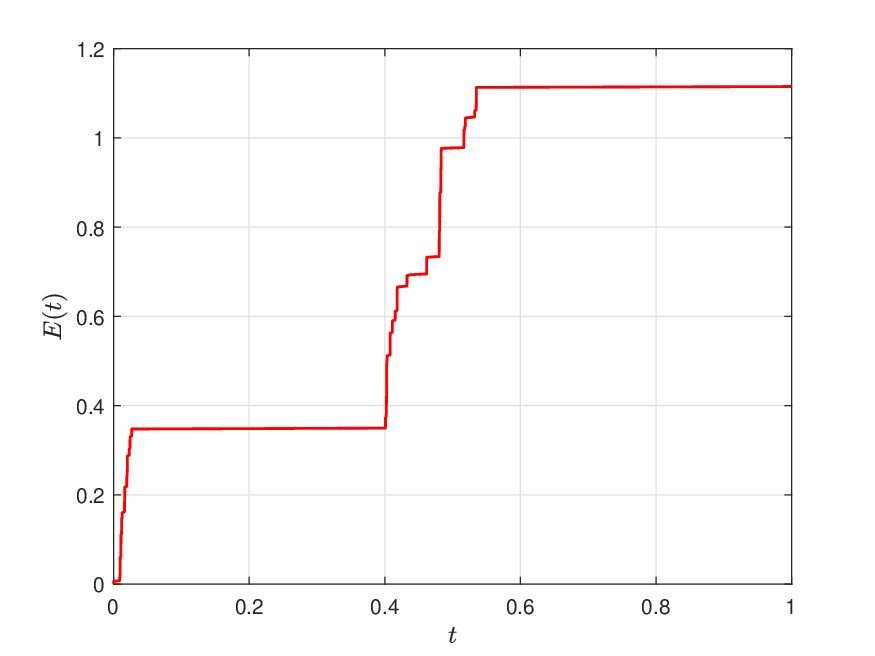}}
      \subfigure[$\alpha = 0.90$]
      {\includegraphics[width=0.45\textwidth]
      {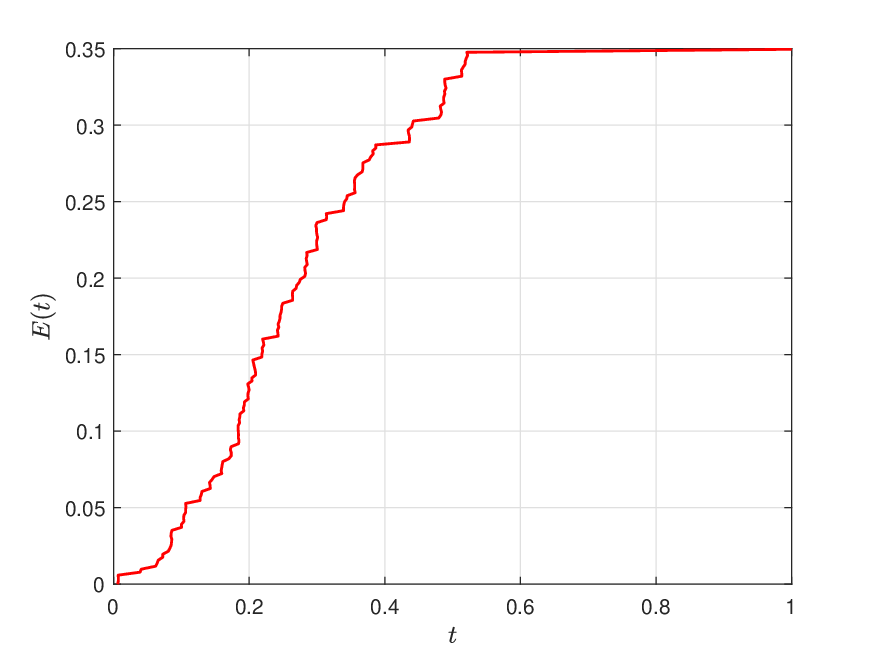}}
      \centering\caption{Sample paths for
      $\{E_ {\Delta}(t)\}_{t \in [0,1]}$
       with $\Delta = 2^{-9}$}
      \label{fig:OneDpII}
\end{figure}

As the stochastic $\theta$ method \eqref{eq:Stmethod} with $\theta \in (0,1]$ applied to the original SDE of \eqref{eq:example} is implicit, the Newton iterations with precision $10^{-5}$ are employed to solve the nonlinear equation arising from the implementation of the underlying method in every time step. Besides, combining \eqref{eq:3.2} and \eqref{eq:525252} leads to
\begin{equation*}
	X_{\Delta}(t)
	=
	\left\{\begin{aligned}
		& Y_{\Delta}(n\Delta) = Y_{n},
		& t \in [D(n\Delta),D((n+1)\Delta)),
		n = 0,1,\cdots, N-1,
		\\
		& Y_{\Delta}(N\Delta) = Y_{N},
		& t \in [D(N\Delta),T],
	\end{aligned}\right.
\end{equation*}
which helps us to plot the sample path of solution $\{ X_{\Delta}(t)\}_{t \in [0,1]}$ to \eqref{eq:example}; see Figure \ref{fig:Onesolutionpath}.

\begin{figure}[!htbp]
\begin{center}
      \subfigure[$\alpha = 0.45,\theta = 0$]
      {\includegraphics[width=0.45\textwidth]
      {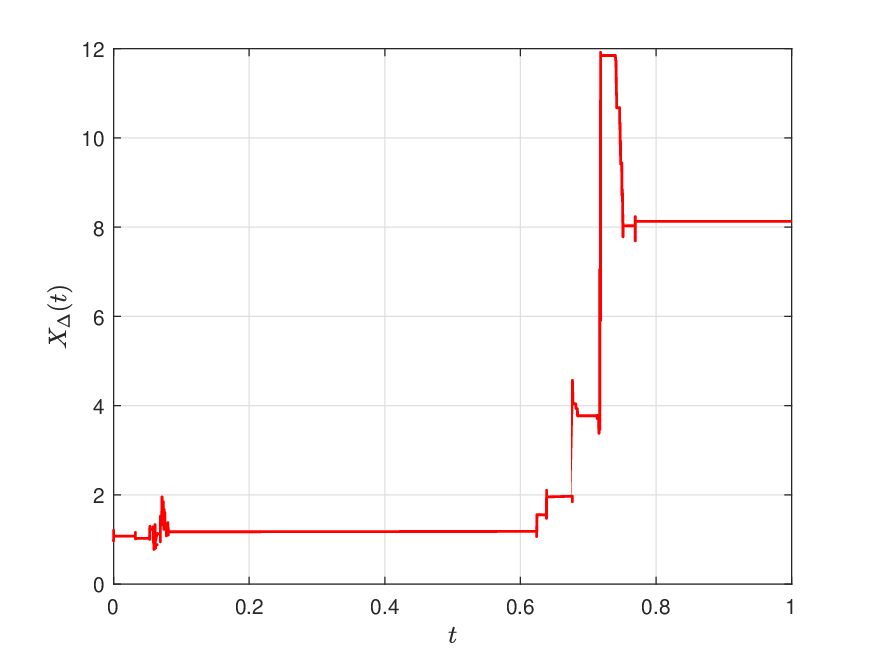}}
      \subfigure[$\alpha = 0.90,\theta = 0$]
      {\includegraphics[width=0.45\textwidth]
      {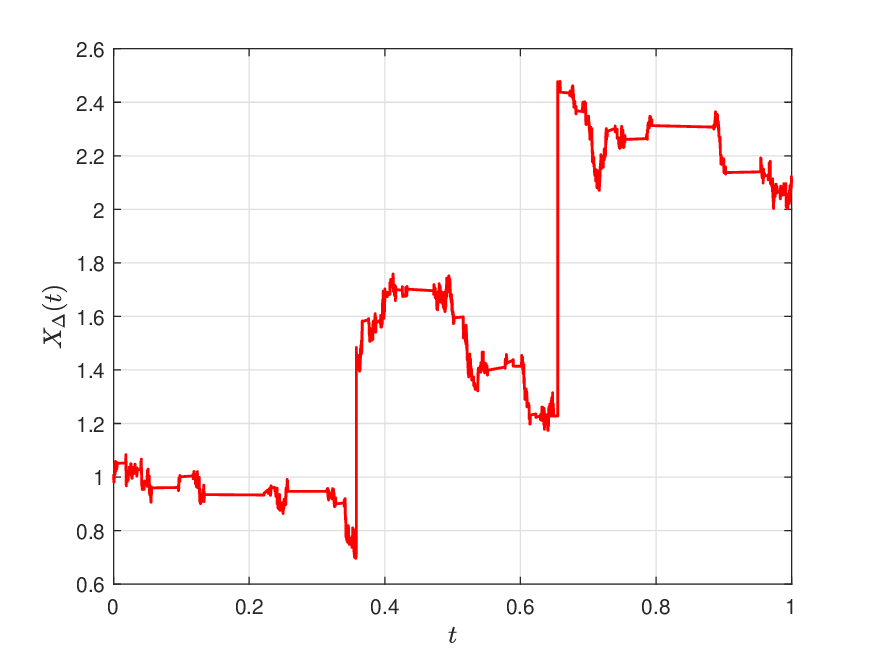}}
      \\
      \subfigure[$\alpha = 0.45,\theta = 0.5$]
      {\includegraphics[width=0.45\textwidth]
      {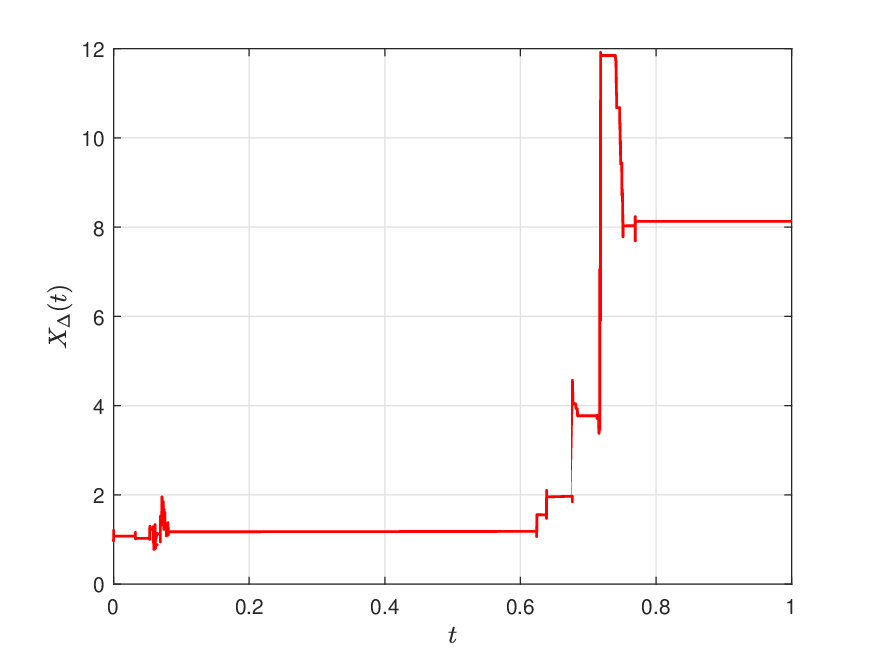}}
      \subfigure[$\alpha = 0.90,\theta = 0.5$]
      {\includegraphics[width=0.45\textwidth]
      {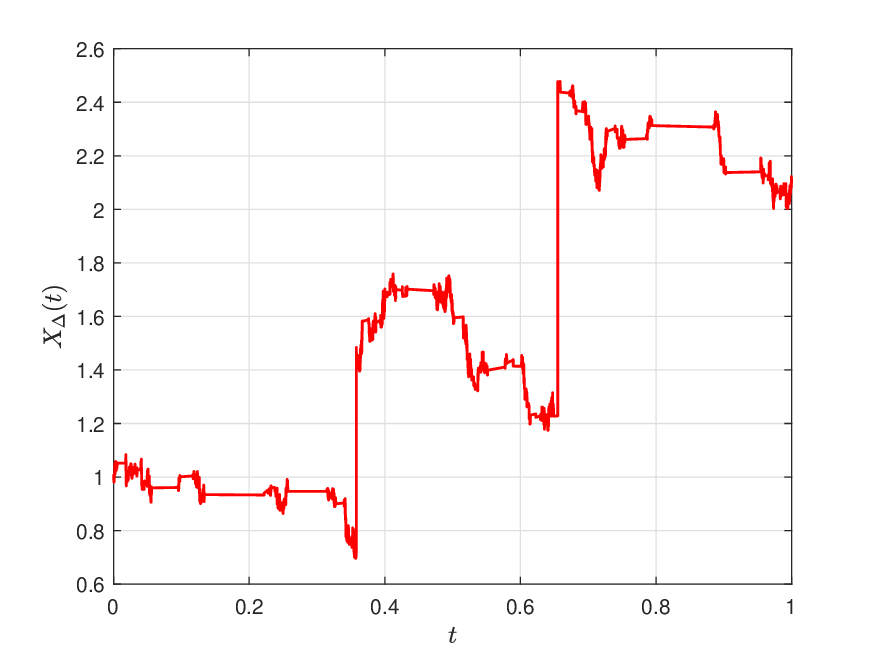}}
      \\
      \subfigure[$\alpha = 0.45,\theta = 1.0$]
      {\includegraphics[width=0.45\textwidth]
      {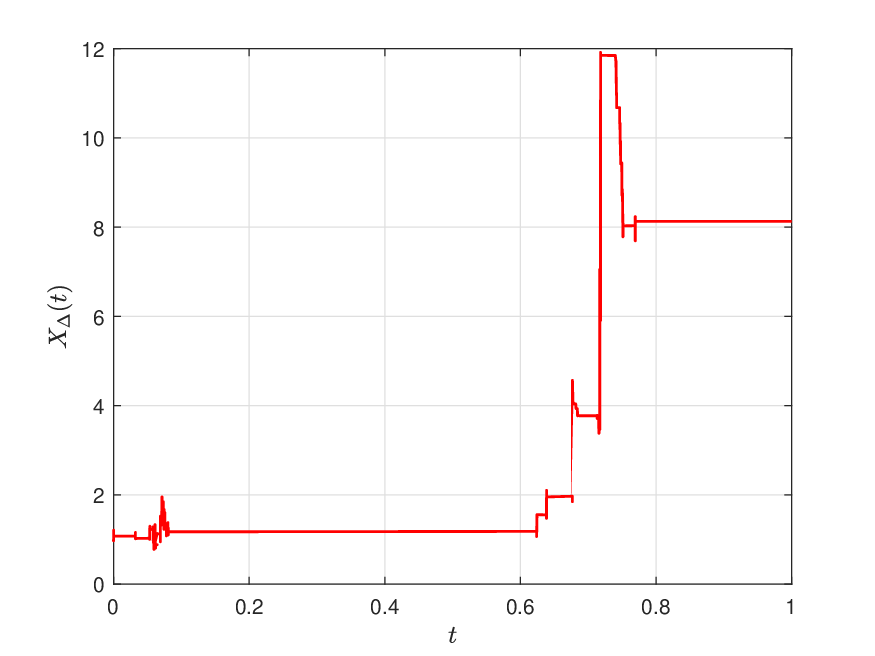}}
      \subfigure[$\alpha = 0.90,\theta = 1.0$]
      {\includegraphics[width=0.45\textwidth]
      {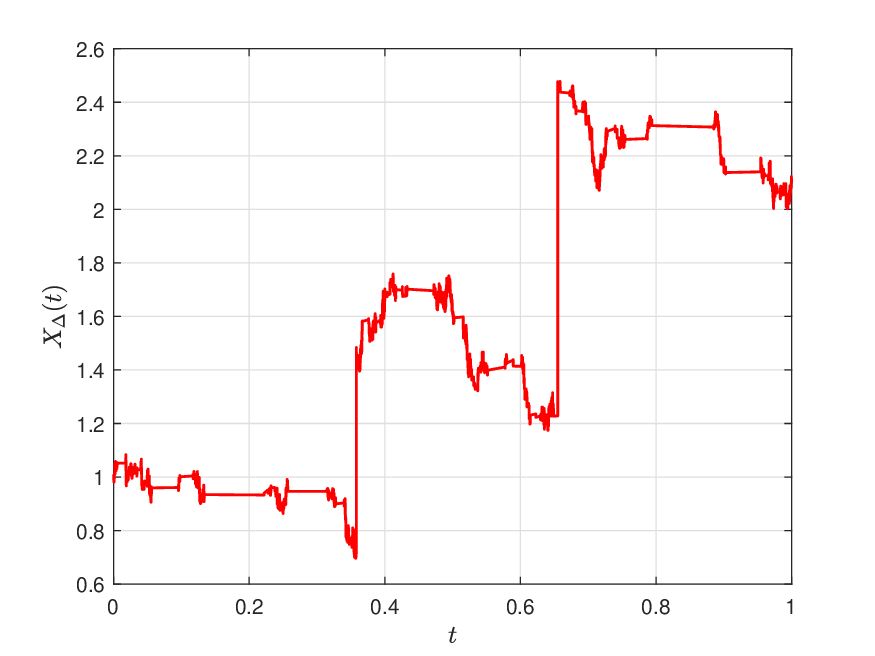}}
      \centering\caption{Sample paths for
      $\{X_{\Delta}(t)\}_{t \in [0,1]}$
      with $\Delta = 2^{-15}$}
      \label{fig:Onesolutionpath}
\end{center}
\end{figure}

To check the strong and weak convergence orders of the numerical approximations for \eqref{eq:example}, we identify the unavailable exact solution with a numerical approximation generated by the considered numerical method with a fine stepsize $\Delta = 2^{-16}$ for Theorem \ref{th:strongorder} and $\Delta = 2^{-12}$ for Theorem \ref{th:weakorder}, respectively. Combining the numerical approximations calculated with $\Delta = 2^{-15}, 2^{-14}, 2^{-13}$ and the expectations approximated by the Monte Carlo simulation with $5000$ different Brownian paths and compensated Poisson process, Figure \ref{fig:strongorder} shows that the root mean square error line and the reference line appear to parallel to each other for different $\theta$, indicating that the mean square convergence order of the stochatic $\theta$ method is $1/2$. Besides, we test the weak convergence order of the Euler--Maruyama method with stepsizes $\Delta = 2^{-10}, 2^{-9}, 2^{-8}$ and $10000$ different Brownian paths and compensated Poisson process. Figure \ref{fig:weakorder} gives the desired result and further supports Theorem \ref{th:weakorder}.

\begin{figure}[!htbp]
	\begin{center}
		\subfigure[$\alpha = 0.45,\theta = 0$]
		{\includegraphics[width=0.45\textwidth]
			{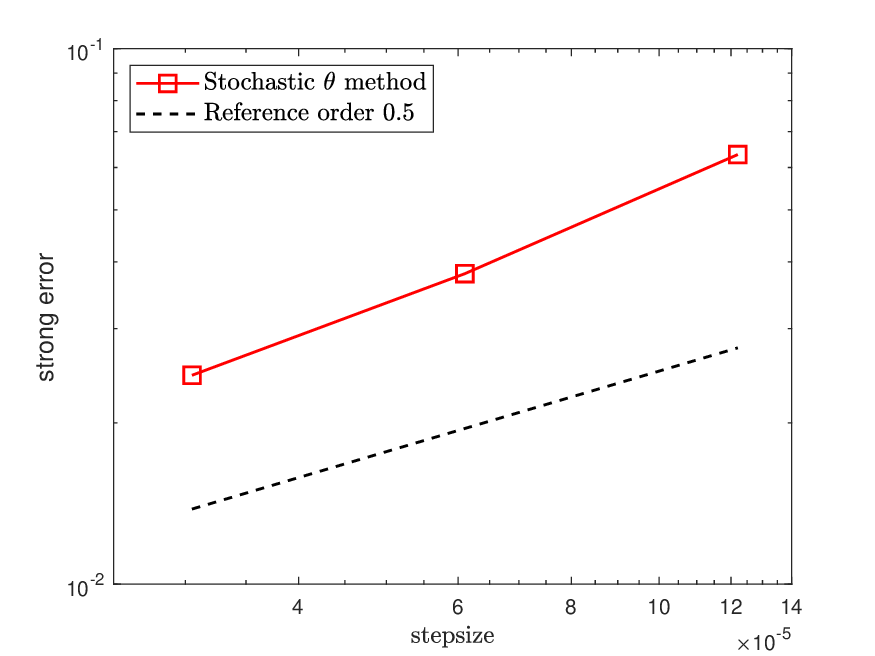}}
		\subfigure[$\alpha = 0.90,\theta = 0$]
		{\includegraphics[width=0.45\textwidth]
			{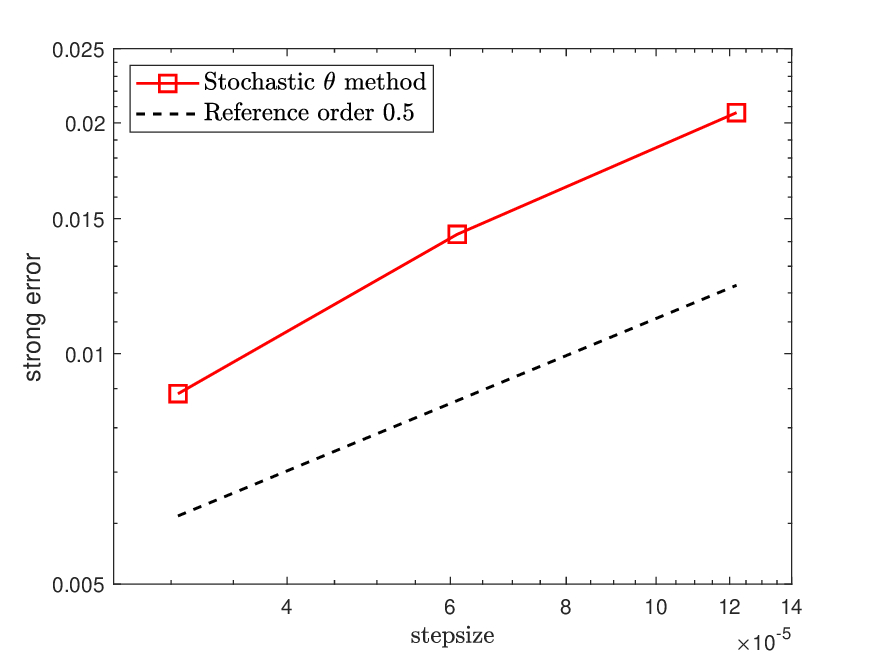}}
		\\
		\subfigure[$\alpha = 0.45,\theta = 0.5$]
		{\includegraphics[width=0.45\textwidth]
			{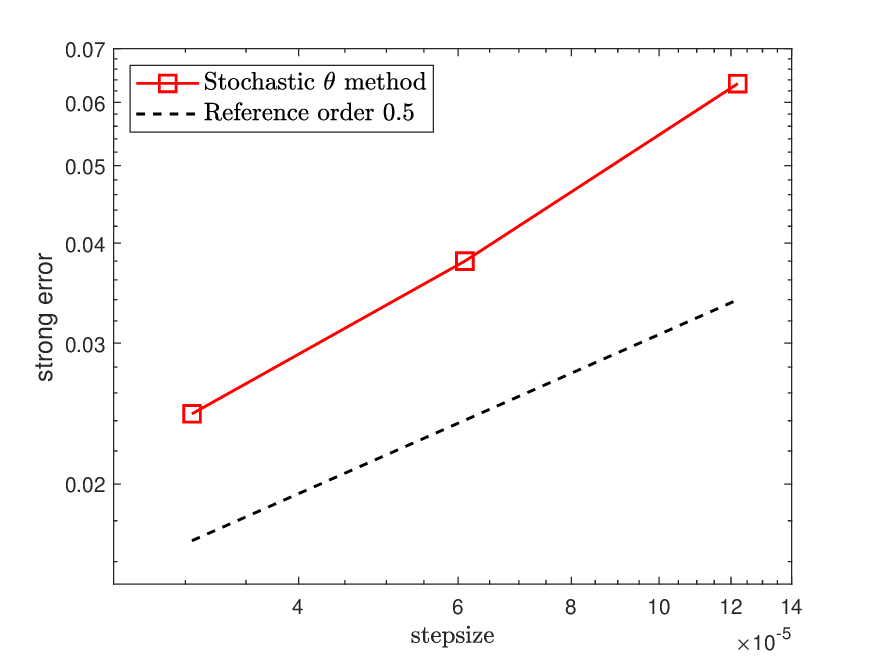}}
		\subfigure[$\alpha = 0.90,\theta = 0.5$]
		{\includegraphics[width=0.45\textwidth]
			{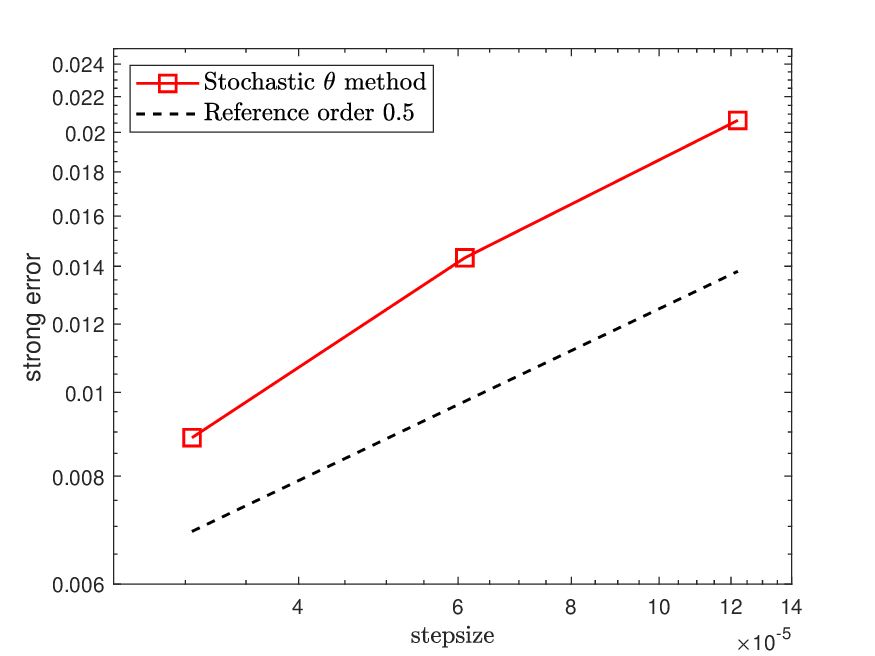}}
		\\
		\subfigure[$\alpha = 0.45,\theta = 1.0$]
		{\includegraphics[width=0.45\textwidth]
			{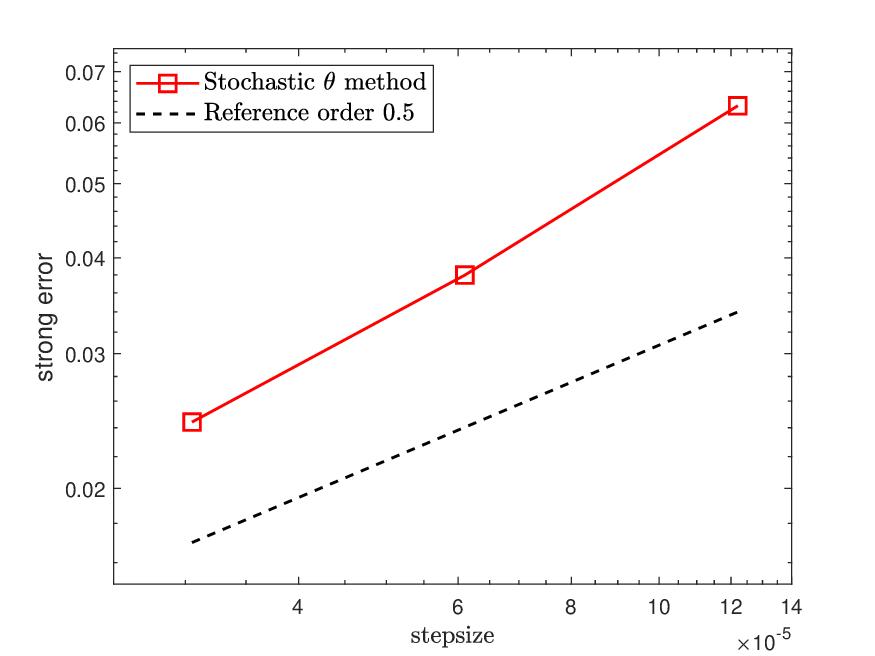}}
		\subfigure[$\alpha = 0.90,\theta = 1.0$]
		{\includegraphics[width=0.45\textwidth]
			{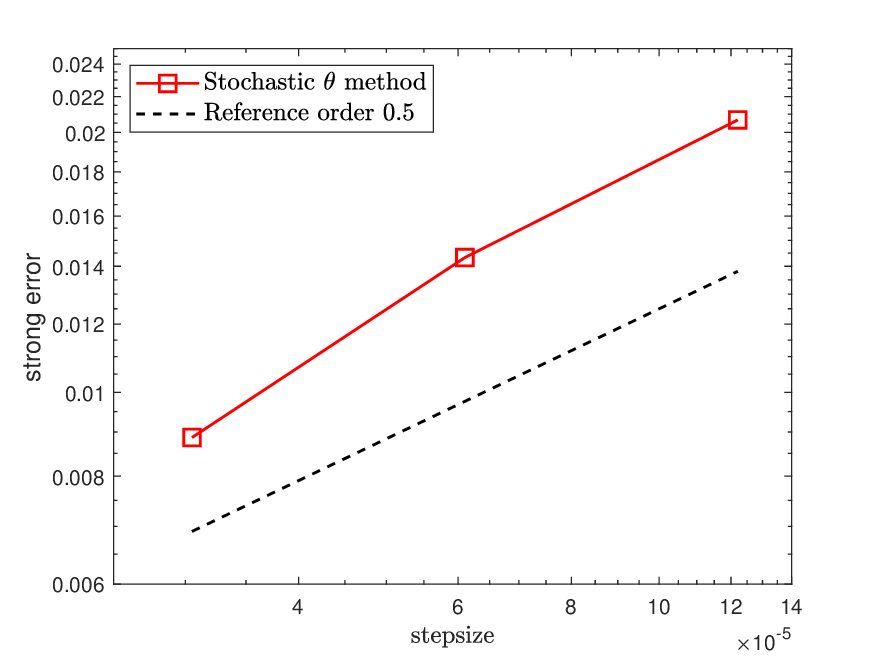}}
		
		\centering\caption{Strong convergence 
        order of stochastic $\theta$ method}
		\label{fig:strongorder}
	\end{center}
\end{figure}

\begin{figure}[!htbp]
\begin{center}
      \subfigure[$\alpha = 0.45,\theta = 0$]
      {\includegraphics[width=0.45\textwidth]
      {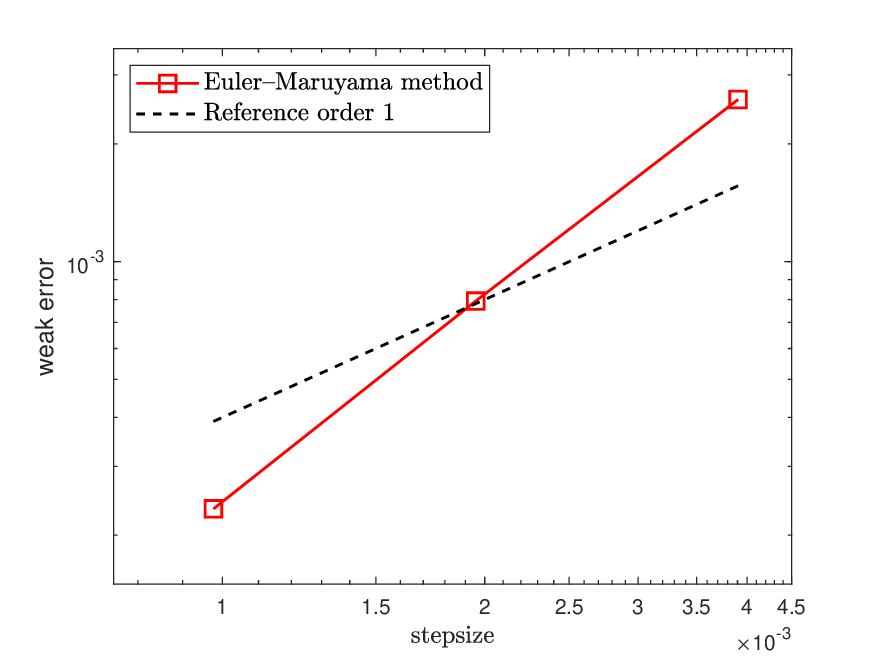}}
      \subfigure[$\alpha = 0.90,\theta = 0$]
      {\includegraphics[width=0.45\textwidth]
      {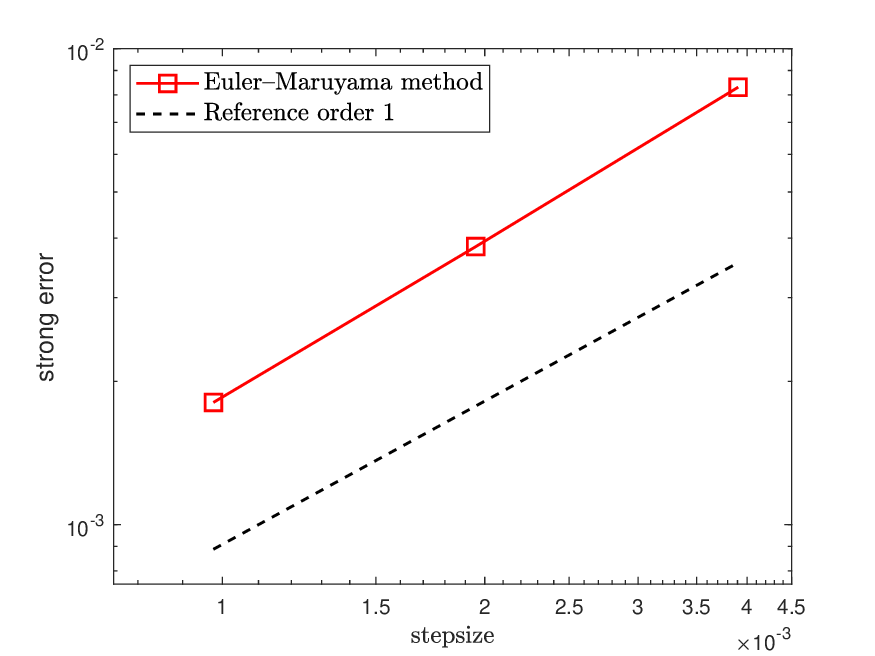}}
      \caption{Weak convergence 
      order of Euler--Maruyama method}
      \label{fig:weakorder}
\end{center}
\end{figure}

\section{Conclusion}			
The strong and weak convergence orders of numerical methods for SDEs driven by time-changed L\'{e}vy noise is investigated in this work. Under the globally Lipschitz conditions, we derive that the stochastic $\theta$ method with $\theta \in [0,1]$ converges strongly with order $1/2$, and the Euler--Maruyama method converges weakly with order $1$.  These theoretical results are finally validated by some numerical experiments. Concerning most SDEs driven by L\'{e}vy noise violating the globally Lipschitz conditions (see \cite{chen2019meansquare, chen2020convergence, dareiotis2016tamed, platen2010numerical} ), we are working on the strong approximations of explicit and implicit numerical methods for SDEs driven by time-changed L\'{e}vy noise with non-globally Lipschitz continuous coefficients. An interesting idea for our future work is to consider the large deviations principles of time-changed SDEs and that of its numerical approximations; see, e.g., \cite{chen2024convergence, hong2024convergence, pacchiarotti2020some} and references therein for more details.

\bibliographystyle{plain}

\end{document}